\documentclass[11pt,a4paper]{article}
\usepackage{preambule}

\newcommand{\Pbf}{\mathbf{P}}
\newcommand{\Ebf}{\mathbf{E}}
\newcommand{\cpc}{\mathrm{cap}}
\newcommand{\TT}{\mathbb{T}}

\title{How thin does random interlacement have to be so that a random walk can see through it?}

\author{Nicolas Bouchot}

\AtEndDocument{\bigskip{\footnotesize%
		\textsc{CEREMADE, Universit\'e Paris Dauphine-PSL, UMR 7534, Place du Maréchal de Lattre de Tassigny, 75775 Paris Cedex 16, France.} \par  
		\textit{E-mail address}: \texttt{nicolas.bouchot@dauphine.psl.eu} }}

\begin{document}
	
	\pagestyle{plain}
	
	\maketitle
	
\begin{abstract}
	\noindent
	The random interlacements $\mathscr{I}(u)$ at level $u$ has been introduced by Sznitman, as a Poissonian collection of independent simple random walk trajectories on $\ZZ^d$, $d\geq 3$, with intensity  $u>0$.
	Since then, several works investigated the properties of the random interlacements intersected with large sets of~$\ZZ^d$. 
	In this paper, we study the asymptotic behavior of the capacity of $\mathscr{I}(u) \cap D_N$, where $D_N$ is the blow up of a compact set $D$, with typical size $N$.
	We determine the correct window $(u_N)_{N\geq 1}$ of the intensity parameter for which the capacity $\cpc(\mathscr{I}(u_N)\cap D_N)$ starts to become negligible compared to $\cpc (D_N)$; this roughly means that a random walk starting from far away starts to see through $\mathscr{I}(u_N)\cap D_N$.
	In the same spirit, we investigate the capacity of the simple random walk conditioned to stay in a large Euclidean ball up to time $t_N$, and find similar asymptotics by taking $t_N = u_N N^d$.
	\\[0.2cm]
\end{abstract}

\section{Introduction}

The random interlacement at level $u>0$ is a random subset $\mathscr{I}(u) \subset \ZZ^d$, $d \geq 3$, which was introduced by Sznitman in~\cite{sznitmanVacantSetRandom2010}, in order to study the  trace of the simple random walk on the torus $\mathbb T_n \defeq \ZZ^d/n\ZZ^d$.
It can be thought as a Poissonian collection of bi-infinite simple random walk trajectories on $\ZZ^d$ (that are transient since $d\geq 3$), with an intensity parameter $u>0$.
We refer to~\cite{drewitzIntroductionRandomInterlacements2014} for a detailed introduction and overview.

In the following, we denote by $\mathbf{P}_x$ the law of a simple random walk $(S_n)_{n\geq 0}$ on $\ZZ^d$, with starting point $x \in \ZZ^d$. We also simply write $\Pbf = \Pbf_0$.
We also denote $K \Subset \ZZ^d$ to say that $K$ is a finite subset of $\ZZ^d$.
Then, a characterization of the law $\PP$ of the random interlacement $(\mathscr{I}(u))_{u > 0}$ is given by the following relation
\begin{equation}
	\label{def:interlacement}
	\forall K \Subset \ZZ^d \,, \quad \PP \big( \mathscr{I}(u) \cap K = \varnothing \big) = \exp \big( -u \, \capN{K} \big) \, .
\end{equation}
Here, $\capN{K}$ denotes the capacity of the set $K$ in $\mathbb Z^d$, defined as
\begin{equation}
	\label{def:capacity}
	\capN{K} \defeq \sum_{x\in K} \bP_x\big( H_K = +\infty \big)\,, 
\end{equation}
with $H_A\defeq \min\{i\geq 1, S_i \in A\}$ the first return time of the set $A$; by convention $\min \varnothing = +\infty$.
The capacity is a non-decreasing function which characterizes the size of a subset seen by a (simple) random walk coming from far away. In this paper, we are interested in how the capacity of a set compares with the capacity of the trace of random walk(s) on this set.

\subsection{Our main question}

We consider a bounded open set $D \subset \RR^d$ and we let $D_N \defeq (N \cdot D) \cap \ZZ^d$ be its (discrete) blow up of size $N$.
We are interested in $\mathscr{I}(u) \cap D_N$ the trace of the random interlacement inside $D_N$, 
and more precisely in the ratio of their capacities
\begin{equation}
	\label{def:ratio}
	\varsigma^{\mathrm{RI}}_{N,u} = \varsigma^{\mathrm{RI}}_{N,u}(D) \defeq \frac{\capN{\mathscr{I}(u) \cap D_N} }{ \capN{D_N}} \,,
\end{equation}
in the regime  $n\to\infty$ and $u\downarrow 0$.
Notice that since $\mathscr{I}(u)\cap D_N \subseteq D_N$, we have that $\varsigma_{n,u}\in [0,1]$.
In particular, we are looking for a condition on the sequence $(u_N)_{n\geq 0}$ for the ratio $\varsigma_{n}$ to converge either to $1$ or $0$, in $\mathbb P$-probability.

Let us define
\begin{equation}\label{eq:def-Theta}
	\Theta_N \defeq
	\begin{cases}
		N & \text{ if } d = 3 \,, \\
		N^2/\log N & \text{ if } d = 4 \,, \\
		N^2 & \text{ if } d \geq 5 \,.
	\end{cases} 
\end{equation}
The motivation behind this definition \eqref{eq:def-Theta} of $\Theta_N$ is so that $\Theta_N$ and $\Ebf \big[ \cpc(\mathcal{R}_{N^2}) \big]$ are of the same order (see Theorem \ref{th:cv-cap-range} below).
We then have the following criterion.

\begin{theorem}\label{th:ratio-cap-entrelac}
	Let $d \geq 3$ and let  $(u_N)_{N\geq 1}$ be a  sequence of positive numbers.
	For any compact regular set $D \subset \RR^d$, we have the following convergences in  $\PP$-probability (and in $L^1$):
	\begin{equation}\label{eq:convergence}
		\varsigma^{\mathrm{RI}}_{N,u_N} = \varsigma_{N,u_N}^{\mathrm{RI}}(D) \defeq \frac{\capN{\mathscr{I}(u_N) \cap D_N}}{\capN{D_N}} \xrightarrow[N \to +\infty]{}
		\begin{dcases}
			0 & \quad \text{ if }\ \lim_{N\to\infty} u_N \Theta_N  = 0 \,, \\
			1 & \quad \text{ if }\ \lim_{N\to\infty} u_N \Theta_N = +\infty \,.
		\end{dcases} 
	\end{equation} 
	Moreover, if $0<\liminf_{N\to\infty} u_N \Theta_N\leq  \limsup_{N\to\infty} u_N \Theta_N <+\infty$, then 
	\begin{equation}
		\label{eq:criticalcase-RI}
		0<\liminf_{N\to\infty} \EE[\varsigma^{\mathrm{RI}}_{N,u_N}]\leq  \limsup_{N\to\infty}  \EE[\varsigma^{\mathrm{RI}}_{N,u_N}] < 1 \,.
	\end{equation}
\end{theorem}

\begin{remark}
	In Theorem \ref{th:ratio-cap-entrelac}, the assumption that $D$ is regular means that $\capR{D} = \capR{\mathrm{Int} (D)}$, where $\capR{\, \cdot \,}$ denotes the Newtonian capacity on $\mathbb R^d$ (see \cite{port2012brownian} for the necessary definitions). It is required to compare $\capR{D}$ with $\capN{D_N}$ (see Proposition \ref{prop:capacite-blowup} below).
\end{remark}

\begin{remark}
	In the case where $\lim_{N\to\infty} u_N \Theta_N = a \in (0,+\infty)$, we expect that the ratio $\varsigma_{N,u_N}$ converges in distribution to a random variable $\varsigma_a \in (0,1)$.
	A first step towards this would be to prove the following tightness-type result:
	for any $\eps>0$ there is some $\eta>0$ such that
	\begin{equation}
		\PP\big( \eta < \varsigma^{\mathrm{RI}}_{N,u_N} < 1-\eta\big) \geq 1-\eps \,.
	\end{equation} 
	The bounds~\eqref{eq:criticalcase-RI} only shows that for sufficiently small $\eta>0$, the probability $\PP\big( \eta < \varsigma_N < 1-\eta\big)$ remains bounded away from $0$.
\end{remark}


\begin{remark}
	\label{rem:holes}
	Let us notice that the parameter $u$ is related to the density (with respect to the volume) of random interlacements in $\ZZ^d$.
	Indeed, for any fixed $u$, we have that 
	\begin{equation}
		\label{eq:def-density}
		\lim_{N\to\infty}  \frac{|\mathscr{I}(u) \cap D_N|}{ |D_N|} = 1-e^{- u / g_0} \eqdef \vartheta(u) \,,
	\end{equation}
	with $g_0 \defeq \sum_{n=0}^{\infty} \mathbf{P}_0(S_n=0)$.
	In particular, when $u_N\downarrow 0$, the random interlacement $\mathscr{I}(u_N)$ is very sparse and has large holes, but the capacity $\cpc (\mathscr{I}(u_N) \cap D_N)$ may remains equivalent to $\cpc (D_N)$, meaning that the holes are not ``visible'' for a random walk arriving from far away.
	In that sense, Theorem~\ref{th:cv-cap-range} gives a threshold on $(u_N)_{n\geq 1}$ for these holes to become ``visible''.
\end{remark}

\paragraph*{About the capacity of a random walk conditioned to stay in a ball}

Random interlacements can be interpreted as an infinite volume measure for the simple random walk on the torus $\TT^d_N = (\ZZ/ N \ZZ)^d$. The link between these two objects has been a powerful tool to tackle problems involving this random walk on the torus (see \cite{teixeiraFragmentationTorusRandom2011,beliusGumbelFluctuationsCover2013,SZN-cylindre}) as well as large-deviation events for the simple random walk on $\ZZ^d$ (see \cite{liLowerBoundDisconnection2014,chiariniLowerBoundsBulk2023}).

The coupling between the two objects presented in \cite{teixeiracoupling} (and previous works) relies on a heuristical argument, which roughly explains that given a subset $K$ of the torus, the random walk may visit $K$ and then stay in $\TT^d_N \setminus K$ a time larger than its mixing time. The eventual return to $K$ would then be independent from the last times the random walk was in $K$, and the number of such returns is well-approximated by a Poisson random variable.

Consider the discrete Euclidean ball $B_N$ with radius $N \geq 1$ and a sequence $(t_N)_{N \geq 1}$, and write $\mathcal{R}_{t_N} = \mathset{S_0 , \, \dots , S_{t_N}}$ the range up to time $t_N$ of the simple random walk on $\ZZ^d$. Under $\Pbf_0( \cdot \, | \, \mathcal{R}_{t_N} \subseteq B_N)$, the random walk behaves in a way such that the heuristical argument still has some truth (see \cite{bouchot2024confinedrandomwalklocally} for a more rigorous statement). By comparison to random interlacements, the range under $\Pbf_0( \cdot \, | \, \mathcal{R}_{t_N} \subseteq B_N)$ can be expected to resemble a random interlacement with parameter $u_N \sim c t_N N^{-d}$. Therefore, we can conjecture an analogue to Theorem \ref{th:ratio-cap-entrelac} for $\mathcal{R}_{t_N}$ under $\Pbf_0( \cdot \, | \, \mathcal{R}_{t_N} \subseteq B_N)$, by plugging $u_N = t_N N^{-d}$. We prove this result in Section \ref{sec:capacite-CRW}, although our method does not rely on any coupling.

\begin{theorem}\label{th:ratio-cap-RW-boule}
	Let $d \geq 3$ and let  $(t_N)_{n\geq 0}$ be a  sequence of positive numbers, under $\Pbf(\cdot \, | \, \mathcal{R}_{t_N} \subseteq B_N)$ we have the following convergences in  $\PP$-probability and in $L^1$:
	\begin{equation}\label{eq:convergence-RW}
		\varsigma_{N,t_N}^{\mathrm{RW}} \defeq \frac{\cpc(\mathcal{R}_{t_N})}{\cpc(B_N)} \xrightarrow[N \to +\infty]{}
		\begin{dcases}
			0 & \quad \text{ if }\ \lim_{N\to\infty} t_N \Theta_N / N^d  =0 \,, \\
			1 & \quad \text{ if }\ \lim_{N\to\infty} t_N \Theta_N / N^d = +\infty \,.
		\end{dcases} 
	\end{equation} 
\end{theorem}

\begin{remark}
	We stated the result in the case of the ball $B_N$, but Theorem \ref{th:ratio-cap-RW-boule} remains valid for \og{}reasonable\fg{} subsets $D_N$ in dimension $d \geq 4$ (see the hypothesis on $D$ in the next chapter). However, we are unable to get the full statement for dimension $d = 3$. Indeed, properly relaxing the random walk takes a time $N^2 \log N$ (to dwarf the next terms in the spectral decomposition, see Lemma \ref{lem:eviter-via-vp}), thus we are not able to handle the case where $N^2 \ll t_N \leq C \log N$ for some constant $C > 0$, which corresponds to $1 \ll t_N \Theta_N N^{-d} \leq C \log N$.
\end{remark}

\begin{remark}
	With our proofs, it is easy to see that an analogue to \eqref{eq:criticalcase-RI} also holds for the random walk under $\Pbf( \cdot \, | \, \mathcal{R}_{t_N} \subseteq B_N)$. However, getting the full statement would require a bit more control on the constants appearing in the proofs, which is not the focus of this paper. We can however prove that there exist $\alpha_- , \alpha_+ \in (0, +\infty)$ such that
	$\varlimsup\limits_{N \to +\infty} \frac{t_N \Theta_N}{N^d} \leq \alpha_-$ implies $\varlimsup\limits_{N \to +\infty} \Ebf \left[ \varsigma_{N,t_N}^{\mathrm{RW}}  \, | \, \mathcal{R}_{t_N} \subseteq B_N \right] < 1$; and similarly if $\varliminf\limits_{N \to +\infty} t_N \Theta_N N^{-d} \geq \alpha_+$, then $\varliminf\limits_{N \to +\infty} \Ebf [\varsigma_{N,t_N}^{\mathrm{RW}}  \, | \, \mathcal{R}_{t_N} \subseteq B_N ] > 0$.
\end{remark}

\paragraph*{Comparison with the volume of the random walk.}

Instead of the capacity, it is natural to study the volume of the random walk conditioned to stay in $B_N$. It is relatively easy to see that in dimension $d \geq 3$, the relative volume of the simple random walk undergoes a phase transition according to the following criterion: under $\Pbf( \cdot \, | \, \mathcal{R}_{t_N} \subseteq B_N)$,
\begin{equation}\label{eq:convergence-RW-volume}
	v_{N,t_N}^{\mathrm{RW}} \defeq \frac{|\mathcal{R}_{t_N}|}{|B_N|} \xrightarrow[N \to +\infty]{\PP}
	\begin{dcases}
		0 & \quad \text{ if }\ \lim_{N\to\infty} t_N N^{-d} = 0 \,, \\
		1 & \quad \text{ if }\ \lim_{N\to\infty} t_N N^{-d} = +\infty \,.
	\end{dcases} 
\end{equation} 
One way to understand this phase transition is through the coupling between the SRW on the torus and random interlacements, or more closely related the coupling presented in \cite{bouchot2024confinedrandomwalklocally}.

One can observe that for $d \geq 5$, the ratio $\varsigma_{N,t_N}^{\mathrm{RW}}$ undergoes a phase transition at $t_N$ of order $N^{d-2}$, while the ratio $v_{N,t_N}^{\mathrm{RW}}$ undergoes a phase transition at $t_N$ of order $N^d$. One can then make two observations.

First, we stress the existence of a phase $N^{d-2} \ll t_N \ll N^d$ where the walk has negligible volume but full capacity, the same way as random interlacements.

Second, Theorem \ref{th:ratio-cap-RW-boule} is part of the large collection of results about the capacity of the simple random walk in dimension $d$ that are analogs of results about its volume in dimension $d-2$. One can see for example the remarks below \cite[Corollary 1.5, Prop 1.6]{asselah-cap-Z4} as well as the works \cite{dvoretzkyProblemsRandomWalk1951,demboCapacityRangeRandom2024,gallProprietesIntersectionMarches1986a}.

\paragraph*{Comparing with independent Bernoulli random variables.}

Let us compare Theorem~\ref{th:ratio-cap-entrelac} to the case where the random interlacement $\mathscr{I}(u)$ is replaced by the set $\mathscr{B}(p) \defeq \{x \in \ZZ^d , \omega_x=1\}$, where $(\omega_x)_{x\in \ZZ^d}$ is a collection of i.i.d.\ Bernoulli random variables with parameter $p = 1 - e^{-u} > 0$ with law also denoted by $\PP$.
In the spirit of Remark~\ref{rem:holes}, the question is to understand whether the ``holes'' are more visible for the random interlacement than for the Bernoulli random variables.

Adapting the proof of Theorem~\ref{th:ratio-cap-entrelac} gives the following result.
\begin{proposition}
	Let $d \geq 3$ and let $(p_N)_{N\geq 1}$ be a sequence of positive numbers and write $p_N = 1 - e^{-u_N}$.
	For any compact regular set $D \subset \RR^d$, we have the following convergences in $\PP$-probability:
	\begin{equation}\label{eq:convergence-ber}
		\kappa_{N,p_N} = \kappa_{N,p_N}(D) \defeq \frac{\cpc(\mathscr{B}(p_N) \cap D_N)}{ \cpc(D_N)} \xrightarrow[N \to +\infty]{}
		\begin{dcases}
			0 & \quad \text{ if }\ \lim_{N\to\infty} u_N N^2  = 0 \,, \\
			1 & \quad \text{ if }\ \lim_{N\to\infty} u_N N^2 = +\infty \,.
		\end{dcases} 
	\end{equation} 
	Moreover, if $0<\liminf\limits_{N\to\infty} u_N N^2 \leq  \limsup\limits_{N\to\infty} u_N N^2 <+\infty$, then 
	\begin{equation}
		\label{eq:criticalcase-ber}
		0<\liminf_{N\to\infty} \EE[\kappa_{N,p_N}]\leq  \limsup_{N\to\infty} \EE[\kappa_{N,p_N}] < 1 \,.
	\end{equation}
\end{proposition}

We therefore notice that in dimension $d=3$ and $d=4$, the threshold on $(p_N)_{N\geq 1}$ for the ratio $\kappa_{N,p_N}$ to vanish is much smaller than for $\varsigma^{\mathrm{RI}}_{N,p_N}$. This means that the ``holes'' of the random interlacement make the capacity $\cpc(\mathscr{I}(p) \cap D_N)$ (much) smaller than $\cpc(\mathscr{B}(p) \cap D_N)$ in dimension $d=3,4$. 
The threshold is the same in dimension $d\geq 5$, but an interesting question would be to determine whether in general we have $\EE[\varsigma^{\mathrm{RI}}_{N,u_N}]\leq \EE[\kappa_{N,p_N}]$, or if $\EE[\varsigma^{\mathrm{RI}}_{N,u_N}]$ is much smaller than $\EE[\kappa_{N,p_N}]$ in the regime $u_N N^2 \to 0$.

\subsection{A motivation from the folding of random walks}

One of our main motivations is the study of folding phenomena for random walks; this arises for instance when considering a random walk among obstacles.
More precisely, consider a random walk $(S_n)_{n\geq 0}$ on $\ZZ^d$, which evolves in a (random) set of obstacles $\mathscr{O}$, with law denoted $\PP$. The random walk is killed when hitting $\mathscr{O}$, so the survival time is $\tau\defeq \min\{i\geq 0, S_i \in \mathscr{O}\}$.
One can then study the probability of surviving for a long time and the properties of the random walk when conditioned on surviving for a long time: one may either consider a \emph{quenched} setting, that is fixing a realization of $\mathscr{O}$, or an \emph{annealed} setting, that is averaging over realizations of $\mathscr{O}$.
The quenched survival probability is $\mathbf{P}(\tau>n) =\mathbf{P}( \mathscr{O} \cap \mathcal{R}_n =\varnothing)$, where  $\mathcal{R}_n= \{S_0,\ldots, S_n\}$ is the range of the simple random walk up to time $n$; the annealed probability is $\PP\mathbf{P}(\mathscr{O} \cap \mathcal{R}_n =\varnothing)$.

There are (at least) two natural sets of obstacles that one could consider: $\mathscr{O}= \mathscr{B}(p)$, a set of independent Bernoulli obstacles of parameter $p \defeq 1-e^{-u}$, or  $\mathscr{O}= \mathscr{I}(u)$, a random interlacment of parameter $u > 0$.
Then, the annealed probabilities of survival are respectively
\begin{equation}
	\label{def:penalized}
	\EE \big[\mathbf{P}(\mathscr{B}(p) \cap \mathcal{R}_n = \varnothing) \big]=
	\mathbf{E}\Big[ e^{- u |\mathcal{R}_n|} \Big]\,,
	\qquad
	\EE \big[\mathbf{P}(\mathscr{I}(u) \cap \mathcal{R}_n =\varnothing) \big]
	=\mathbf{E}\Big[ e^{- u \cpc(\mathcal{R}_n)} \Big]\,,
\end{equation}
using~\eqref{def:interlacement} in the case of interlacements.
Both survival probabilities can also be interpreted as partition functions of a Gibbs measure penalizing a random walk $(S_i)_{0\leq i\leq n}$ by its volume, resp.\ by its capacity.
Studying these annealed probabilities then amounts to obtaining large deviation principles for the volume, resp.\ the capacity, of the range $\mathcal{R}_n$ of a random walk.

\paragraph*{Random walk penalized by its volume.}
The first model of~\eqref{def:penalized}, \textit{i.e.}\ the random walk penalized by the size of its range, is by now well-understood, starting from the seminal work of Donsker and Varadhan~\cite{donskerNumberDistinctSites1979}.
It is now known that the random walk folds into a ball of radius $\rho_n \defeq c_d (n/\beta)^{1/(d+2)}$ and fills it, see~\cite{dingGeometryRandomWalk2020,berestyckiRandomWalkPenalised2021} for precise statements.

Let us briefly present the heuristic for determining the radius $\rho_n$. 
The probability that a random walk remains up to time $n$ in a ball of radius $r$ is approximately $\exp(- \lambda_d n/r^2)$ (with $\lambda_d$ the principal eigenvalue of the Laplacian in the unit ball). On the other hand, the volume of that ball is $\pi_d r^{d}$ (with $\pi_d$ the volume of the unit ball).
Optimizing over the entropic cost $\lambda_d n/r^2$  and the volume penalty $\beta \pi_d r^d$ gives the correct radius $\rho_n \defeq c_d  (n/\beta)^{1/(d+2)}$ of the ball that the random walk effectively folds into.

Let us mention that the literature is mostly concerned with the case of a constant parameter $\beta>0$, but Donsker and Varadhan's argument should also be valid if $\beta=\beta_n \downarrow 0$, as long as $\rho_n \ll n^{1/d}$, \textit{i.e.}\ $\beta_n \gg n^{-2/d}$.
Indeed, a random walk constrained to remain in a ball of radius $\rho_n$ will fill almost completely this ball as long as its volume is much smaller than $n$, that is as long as $\rho_n \ll n^{1/d}$.
The case when $\beta_n$ is of order $n^{1/d}$ actually falls into the large deviation regime studied in \cite{van2001moderate}, known as the \textit{Swiss cheese} picture: the random walk will stay on a scale $n^{1/d}$, leaving large holes in its range.

\paragraph*{Random walk penalized by its capacity.}

As far as the second model in~\eqref{def:penalized} is concerned, \textit{i.e.}\ the random walk penalized by the capacity of its range, we are not aware of any result.
The large deviations results obtained in~\cite{asselah:hal-01832098} for $\cpc(\mathcal{R}_n)$ do not give a sharp large deviation principle and actually do not treat the scale needed to estimate the right-hand side of~\eqref{def:penalized}.

It is however believed that also in that case, the random walk should fold into a ball of radius~$\rho_n$.
One can repeat the heuristic argument from the above section, replacing the volume of the ball of radius $r$ by its capacity $\varpi_d r^{d-2}$ (with $\varpi_d$ the capacity of the unit ball).
Optimizing over the entropic cost $\lambda_d n/r^2$  and the capacity penalty $u \varpi_d r^{d-2}$ gives the correct radius $\rho_n \defeq c'_d  (n/u)^{1/d}$ of the ball that the random walk should effectively fold into.

Let us stress that here, for a fixed $u$, the walk constrained to stay in a ball of radius $n^{1/d}$ should not fill completely that ball, leaving some holes; however its capacity should still be very close to the capacity of the ball.
In the same spirit as in the previous paragraph, the heuristic argument should hold as long as the capacity of a random walk constrained to stay in a ball of radius $\rho_n$ is equivalent to the capacity of the ball. According to Theorem \ref{th:ratio-cap-RW-boule} this should hold as long as $\rho_n \ll \varrho_n$ with
\begin{equation}
		\varrho_n \defeq
	\begin{cases}
		N & \text{ if } d = 3 \,, \\
		(n/\log n)^{1/2} & \text{ if } d = 4 \,, \\
		n^{1/(d-2)} & \text{ if } d \geq 5 \,.
	\end{cases} 
\end{equation}

The argument breaks down when $\rho_n$ is of order $\varrho_n$, \textit{i.e.}\ when holes appears in the range of the random walk (from the point of view of the capacity, see Remark~\ref{rem:holes}).
The corresponding large deviation results for $\cpc \big(\mathcal{R}_n \big)$ falls into the scope of the work (in progress) \cite{BBH-capacity}, which describes an \textit{Italian spaghetti} picture: in dimension $d\geq 5$, the random walk should stay on a scale $n^{1/(d-2)}$, leaving ``holes for the capacity''.

\subsection{Some other related works}

The intersection $\mathscr{I}(u) \cap D_N$ of the interlacement with a blow-up of a compact has already been studied in the literature, see~\cite{liLargeDeviationsOccupation2015,sznitmanBulkDeviationsLocal2023}.
Let us mention in particular that,  for a fixed parameter $u>0$, large deviation have been obtained for its volume $|\mathscr{I}(u) \cap D_N|$ and of its capacity $\cpc(\mathscr{I}(u) \cap D_N)$, in particular in order to study probabilities of disconnection events by the random interlacement, see for instance~\cite{sznitmanDisconnectionRandomWalks2017a}.

As far as the volume $|\mathscr{I}(u) \cap D_N|$ is concerned, \cite{sznitmanBulkDeviationsLocal2023} obtains the following large deviation result for the covering ratio
\[
\nu_{N,u} \defeq \frac{|\mathscr{I}(u) \cap D_N|}{|D_N|} \,.
\]
Recall that $\vartheta(u)\defeq 1 - e^{-u/g_0}$ is defined in~\eqref{eq:def-density}. Then for any $1 > \nu > \vartheta(u)$, \cite[Theorem~6.1]{sznitmanBulkDeviationsLocal2023} obtains
\[ 
\lim_{N \to +\infty} \frac{1}{N^{d-2}} \proba{\nu_{N,u} > \nu} = - \inf \mathset{\frac{1}{2d} \int_{\RR^d} |\nabla \varphi|^2 \, : \, \varphi \in \mathscr{C}_0^\infty(\RR^d), \int_D \vartheta\big((\sqrt{u} + \varphi(z))^2 \big) \, \dd z > \nu |D|} \, .  
\]

For the capacity  $\cpc(\mathscr{I}(u) \cap D_N)$, \cite{liLargeDeviationsIntersection2024} obtained a large deviation principle for the ratio $\varsigma^{\mathrm{RI}}_{N,u}$ defined in~\eqref{def:ratio}, in the case where $D$ is the unit ball of $\RR^d$: for any $u > 0$ and $\lambda\in (0,1)$,
\[
\lim_{N \to +\infty} \frac{1}{N^{d-2}} \proba{ \varsigma^{\mathrm{RI}}_{N,u} < \lambda} = -\frac{u}{d} \inf_{A \in \Sigma_{\lambda}} \capR{D\setminus A} \, , 
\]
where $\Sigma_{\lambda} \defeq \{ A \subset D \, : \, \capR{A} \leq \lambda \capR{D} , \text{ $A$ is a finite union of boxes} \}$.
In particular, it shows that for a fixed $u>0$, the ratio of the capacities $\varsigma^{\mathrm{RI}}_{N,u}$ goes to $1$ in $\PP$-probability. In the present paper (that is Theorem \ref{th:ratio-cap-entrelac}) we tackle a somehow different question, asking how small $u$ can be to still have a ratio of order $1$.

\section{Random interlacements: Proof of Theorem~\ref{th:ratio-cap-entrelac}}\label{sec:entrelac}

\subsection{Preliminaries on capacities}

Let us start by rewriting the ratio $\varsigma_{n,u}$ defined in~\eqref{def:ratio}.
Recalling the definition~\eqref{def:capacity}, we define for $K\Subset \ZZ^d$ the equilibrium measure of $K$ as 
\[
e_K(x) \defeq \indic{x \in K} \mathbf{P}_x (H_K = +\infty) \,,
\]
so in particular $\cpc(K) = e_K(\ZZ^d) = e_K(\partial K)$.
Define also the normalized equilibrium measure of~$K$, or harmonic measure with respect to $K$, as $\bar e_K \defeq \frac{1}{\cpc(K)} e_K$, which is supported on the (interior) boundary of $K$.

If we denote $g(x,y) \defeq \sum_{k = 0}^{+\infty} \probaRW{x}{S_k = y}$ the Green function of the random walk, the equilibrium measure (and the capacity) is linked to $g(x,y)$ through the last exit decomposition (see e.g.~\cite[Lem.~2.12]{drewitzIntroductionRandomInterlacements2014}):
\begin{equation}
	\label{eq:decomp-derniere-sortie}
	\probaRW{z}{H_K < +\infty} = \sum_{y \in K} g(z,x) e_K(x)  \, .
\end{equation}
In turn, we also have another characterization of the capacity of a set $K$, as the properly normalized probability that $K$ is hit by a simple random walk starting from far away.
The harmonic measure $\bar e_K$ can then be interpreted as the distribution of the hitting point $S_{H_K}$ for a simple random walk starting from far away, conditioned to hit $K$. More precisely, we have
\begin{equation}
	\label{defbis:capacity}
	\cpc(K) = \lim_{|z|\to\infty} \frac{1}{g(0,z)} \bP_{z}\big( H_K<+\infty \big) \,,
	\quad
	\bar e_K(x) = \lim_{|z|\to\infty} \bP_z\big(S_{H_K}=x \mid H_K<+\infty \big) \,.
\end{equation}
One can find the proof of these results in \cite[Section 6.5]{lawlerRandomWalkModern2010}.

Our starting point for the rest of the proof will be the following formula for the difference of two capacities: if $A \subseteq B \Subset \ZZ^d$ and $\cpc(B) > 0$, then we have
\begin{equation}
	\label{eq:diff-cap-inclusion}
	\capN{B} - \capN{A} = \capN{B} \sum_{x\in B} \bar e_B(x) \bP_x\big( H_A = +\infty \big) \, .
\end{equation}
Indeed, starting from~\eqref{defbis:capacity}, we get that
\[
\begin{split}
	\capN{B} - \capN{A} 
	& = \lim_{|z|\to\infty} \frac{1}{g(0,z)} \big( \bP_{z}(H_B<+\infty) - \bP_z(H_A<+\infty) \big)  \\
	& = \lim_{|z|\to\infty} \frac{1}{g(0,z)}  \bP_{z}(H_B<+\infty)  \bP_z(H_A=+\infty \mid H_B <+\infty) \,,
\end{split}
\]
where we have used that $A\subseteq B$ to get that $\bP_y(H_B<+\infty)-\bP_y(H_A<+\infty) = \bP_y(H_B<+\infty, H_A =+\infty)$.
Decomposing the last probability as
\[
\bP_z(H_A=+\infty \mid H_B <+\infty) = \sum_{x \in B} \bP_z(S_{H_B}=x \mid H_B <+\infty) \bP_x(H_A=+\infty) \,,
\]
we deduce~\eqref{eq:diff-cap-inclusion} from~\eqref{defbis:capacity}.

\subsection{Main argument and key proposition}\label{ssec:ratio-cap-key-argument}

Let us write $\bar e_N = \bar e_{D_N}$ for simplicity, then \eqref{eq:diff-cap-inclusion} yields
\[
1- \varsigma^{\mathrm{RI}}_{N,u_N} = \sum_{x\in D_N} \bar e_{N}(x) \mathbf{P}_x\big( H_{\mathscr{I}(u_N) \cap D_N} = +\infty \big)  = \sum_{x\in D_N} \bar e_{N}(x) \mathbf{P}_x\big( \mathcal{R}_{\infty}\cap \mathscr{I}(u_N) \cap D_N = \varnothing \big) \,,
\]
where $\mathcal{R}_{\infty} = \{S_0,S_1,S_2,\ldots\}$ is the range of the random walk $(S_n)_{n\geq 0}$ in infinite time-horizon.
Since $\varsigma^{\mathrm{RI}}_{N,u_N}\in [0,1]$, in order to prove~\eqref{eq:convergence}, we only need to control
\begin{equation}
	\label{eq:rapport-cap-exp}
	\begin{split}
		\esp{1-\varsigma^{\mathrm{RI}}_{N,u_N}} &= \sum_{x\in D_N} \bar e_{N}(x) \EE\left[ \mathbf{P}_x\big( \mathcal{R}_{\infty}\cap \mathscr{I}(u_N) \cap D_N = \varnothing \big) \right] \\
		&= \sum_{x\in D_N} \bar e_{N}(x)  \espRW{x}{e^{-u_N \capN{\mathcal{R}_\infty \cap D_N}}} \,,
	\end{split}		 
\end{equation}
where we have used the characterization~\eqref{def:interlacement} of the law of the random interlacement.

Then, we use the following key result, which estimates the capacity $\capN{\mathcal{R}_{\infty} \cap D_N}$ when the random walk starts from a point drawn from $\bar e_N$.
Recall the definition~\eqref{eq:def-Theta} of $\Theta_N$.

\begin{proposition}\label{lem:deviation-cap-infty}
	For any $\eta > 0$, there is some $\delta(\eta)$ with $\lim_{\eta\downarrow 0} \delta(\eta) = 0$ such that
	\[ 
	\limsup_{N \to +\infty} \sum_{x\in D_N} \bar e_N(x)\probaRW{x}{\capN{\mathcal{R}_\infty \cap D_N} \not\in [\eta, \tfrac{1}{\eta}] \Theta_N} \leq \delta(\eta) \, .
	\]
\end{proposition}

This proposition allows us to conclude the proof of Theorem~\ref{th:cv-cap-range}.
Indeed, we have the following bound:
\[
\espRW{x}{e^{-u_N \capN{\mathcal{R}_\infty \cap D_N}}} \geq e^{-\frac1\eta  u_N \Theta_N}   \mathbf{P}_x\big(\capN{\mathcal{R}_\infty \cap D_N} < \tfrac1\eta \Theta_N\big) \,.
\]
If we let $\ell_1\defeq \limsup_{n\to\infty} u_N \Theta_N$ then thanks to Proposition~\ref{lem:deviation-cap-infty}, 
we get that for any $\eta>0$, 
\[
\liminf_{n\to\infty} \esp{1-\varsigma^{\mathrm{RI}}_{N,u_N}} \geq e^{- \frac1\eta \ell_1} (1- \delta(\eta) ) \,,
\]
or in other words
\[
\limsup_{n\to\infty} \esp{\varsigma^{\mathrm{RI}}_{N,u_N}} \leq 1-e^{- \frac1\eta \ell_1} (1- \delta(\eta) ) \,.
\]
In particular, if $\ell_1 =\lim_{N\to\infty} u_N \Theta_N = 0$, then since $\delta(\eta)$ can be made arbitrarily small, this shows that $\esp{\varsigma^{\mathrm{RI}}_{N,u_N}}$ goes to $0$ as $N\to\infty$, \textit{i.e.}\ the second line in~\eqref{eq:convergence}.
Otherwise, it proves the upper bound in~\eqref{eq:criticalcase-RI}.

For the other bound, we use a similar idea:
\[
\begin{split}
	\espRW{x}{e^{-u_N \capN{\mathcal{R}_\infty \cap D_N}}} &\leq e^{-\eta u_N \Theta_N} \mathbf{P}_x\big(\capN{\mathcal{R}_\infty \cap D_N}  > \eta \Theta_N\big)  +\mathbf{P}_x\big(\capN{\mathcal{R}_\infty \cap D_N} \leq  \eta\Theta_N\big)  \\
	& = 1- \mathbf{P}_x\big(\capN{\mathcal{R}_\infty \cap D_N} >  \eta\Theta_N\big) (1- e^{-\eta u_N \Theta_N}) \,.
\end{split}
\] 
If we let $\ell_2\defeq \liminf_{N\to\infty} u_N \Theta_N$, then thanks to Lemma~\ref{lem:deviation-cap-infty} we get that for any $\eta>0$, 
\[
\limsup_{N\to\infty} \esp{1-\varsigma^{\mathrm{RI}}_{N,u_N}} \leq 1-  (1- \delta(\eta))(1-e^{-\eta \ell_2}) \,,
\] 
or put otherwise
\[
\liminf_{N\to\infty} \esp{\varsigma^{\mathrm{RI}}_{N,u_N}} \geq (1-\delta(\eta))  (1-e^{-\eta \ell_2}) \,.
\]
In particular, if $\ell_2 =\lim_{N\to\infty} u_N \Theta_N = +\infty$, then since $\delta(\eta)$ is arbitrarily small we obtain that $\esp{\varsigma^{\mathrm{RI}}_{N,u_N}}$ goes to $1$, which proves the first line in~\eqref{eq:convergence}.
Otherwise, it proves the lower bound in~\eqref{eq:criticalcase-RI}.
\qed

\subsection{Estimates on \texorpdfstring{$\capN{\mathcal{R}_{\infty}\cap D_N}$}{}: proof of Proposition~\ref{lem:deviation-cap-infty}}\label{ssec:control-cap-range-infty-RI}

Before we prove Proposition~\ref{lem:deviation-cap-infty}, let us collect useful results on the capacity of the range of the simple random walk up to time $n$, that we denote $\mathcal{R}_n\defeq \{S_0,S_1,\ldots, S_n\}$.
In particular, the law of large numbers for $\capN{\mathcal{R}_n}$ has been proved in dimension $d\geq 5$ in~\cite{jainRangeRandomWalk1968} (see also~\cite{asselahCapacityRangeRandom2018} for central limit theorems) and in dimension $d=4$ by~\cite{asselah-cap-Z4}; in dimension $d=3$, the convergence in distribution of the rescaled $\capN{\mathcal{R}_n}$ has been proven in \cite{Chang-cap-dim-3-4}.

\begin{theorem}\label{th:cv-cap-range}
	In dimension $d\geq 4$, we have the following law of large numbers: there is $\alpha_d>0$ such that $\mathbf{P}$-almost surely
	\begin{equation}
		\label{eq:cv-cap-range-4-plus}
		\begin{split}
			\frac{1}{n} \capN{\mathcal{R}_n}& \xrightarrow[n \to +\infty]{} \alpha_d  \quad  \text{ if } d\geq 5 \, , \\
			\frac{\log n}{n} \capN{\mathcal{R}_n}& \xrightarrow[n \to +\infty]{} \frac{\pi^2}{8}  \quad  \text{ if } d=4 \,.
		\end{split}
	\end{equation}

	\noindent
	In dimension $d=3$, we have the following convergence in distribution: 
	\begin{equation}
		\label{eq:cv-cap-range-3}
		\frac{1}{\sqrt{n}} \capN{\mathcal{R}_n} \xrightarrow[t \to +\infty]{(d)} \frac{1}{3\sqrt{3}} \, \capR{W_{[0,1]}} \eqdef \mathscr{C}_W \, ,
	\end{equation}
	where $W_{[0,1]}$ is the trace of a three-dimensional Brownian motion up to time $1$.
\end{theorem}

\subsubsection{Probability that \texorpdfstring{$\capN{\mathcal{R}_{\infty}\cap D_N}$ is large}{}}

We now prove the first part of Proposition~\ref{prop:capacite-blowup}.
First of all, since $D$ is compact, we can find some $r>0$ such that $D_N \leq B_{rN}$, where $B_{rN}$ is a ball of radius $rN$ in $\ZZ^d$; only to simplify notation, let us assume that $r=1$.

Therefore, we have that $\mathcal{R}_{\infty} \cap D_N \subseteq \mathcal{R}_{\infty} \cap B_N$, and we will bound
\[
\sup_{x\in B_N} \mathbf{P}_x \big( \capN{\mathcal{R}_{\infty} \cap B_N} \geq \tfrac{1}{\eta} \Theta_N \big) \,.
\]
By rotational invariance, we only need to teat the case of a generic $x \in B_N$, without worrying about the supremum.
For this, we use a last exit decomposition: letting $L_N \defeq \sup\{i\geq 0, S_i \in B_N\}$ the time of last visit to the ball $B_N$, using that $\mathcal{R}_{\infty} \cap B_N \subseteq \mathcal{R}_{L_N}$, we have
\[
\mathbf{P}_x \big( \capN{\mathcal{R}_{\infty} \cap B_N} \geq \tfrac{1}{\eta} \Theta_N \big) \leq \mathbf{P}_x \big( \capN{\mathcal{R}_{A N^2}} \geq \tfrac{1}{\eta} \Theta_N \big) + \mathbf{P}_x \big( L_N > A N^2\big)\,,
\]
where $A=A(\eta) >1$ is some fixed parameter.

For the first term, by the law of large numbers for the capacity, recalling the definition~\eqref{eq:def-Theta} of $\Theta_N$, we have that:
\begin{itemize}
	\item In dimension $d\geq 4$, if $A = c_d/\eta$ with $c_d > 1/\alpha_d$ if $d\geq 5$ and $c_d>16/\pi^2$ if $d=4$, then 
	\[
	\lim_{N\to\infty} \mathbf{P}_x \big( \capN{\mathcal{R}_{A N^2}} \geq \tfrac{1}{\eta} \Theta_N \big) =0 \,.
	\]
	\item In dimension $d=3$, we have
	\[
	\lim_{N\to\infty} \mathbf{P}_x \big( \capN{\mathcal{R}_{A N^2}} \geq \tfrac{1}{\eta} \Theta_N \big) = \mathbf{P}\big( \mathscr{C}_W \geq  \tfrac{1}{\eta\sqrt{A}}\big) \,.
	\]
\end{itemize}

It remains to control the other term. By Donsker's invariance principle, we have that 
\[
\limsup_{N\to\infty} \mathbf{P}_x \big( L_N > A N^2\big) \leq \mathbf{P}_0\Big( \inf_{t>\sqrt{A}} \|W_t\| \leq \sqrt{d} \Big) \,,
\]
where $(W_t)_{t\geq 0}$ is a standard $d$-dimensional Brownian motion (the factor $\sqrt{d}$ comes from the variance of the simple random walk); $\mathbf{P}_x$ also denotes it law, with starting point $x\in \RR^d$.
In the above, we have also used that the probability $ \mathbf{P}_x( L_N > A N^2)$ is maximized at $x=0$.
Now, if we let $\tau_r\defeq \inf\{ s >0 , \|W_s\|=r\}$ be the hitting time of the sphere of radius $r$, we get that for any $r>1$
\[
\mathbf{P}_0\Big( \inf_{t>\sqrt{A}} \|W_t\| \leq \sqrt{d} \Big)\leq \mathbf{P}_0\big(\tau_r< \sqrt{A}/2 \big) \sup_{|z| = r} \mathbf{P}_{z} \big( \tau_{\sqrt{d}} < +\infty \big) + \mathbf{P}_0 \big(\tau_r> \sqrt{A}/2 \big) \,.
\]
Now, by standard results on Brownian motions, for $|z| = r$ we have that $\mathbf{P}_{z} \big( \tau_1 < +\infty \big) = (r/\sqrt{d})^{2-d}$ in dimension $d\geq 3$.
Using that the hitting time of $1$ is larger for the one-dimensional Brownian, we get thanks to the reflection principle that
\[
\mathbf{P}_0\big(\tau_r \geq  \sqrt{A}/2 \big) 
= \mathbf{P}_0\big(\tau_1 \geq  \sqrt{A}/(2\sqrt{r}) \big) \leq  2 \mathbf{P}\big( Z \geq \sqrt{A}/(2\sqrt{r})  \big) \,,
\]
where $Z\sim \mathcal{N}(0,1)$.
All together, we obtain that
\[
\mathbf{P}_0\Big( \inf_{t>\sqrt{A}} \|W_t\| \leq \sqrt{d} \Big) \leq  \inf_{r>1} \Big\{ (r/\sqrt{d})^{2-d} + 2 e^{-\frac{A}{r}} \Big\} \leq C (A / \log A)^{2-d} \,.
\]

We have therefore shown the following.
In dimension $d\geq 4$, choosing $A =c_d/\eta$ with a sufficiently large constant $c_d$,
\[
\limsup_{N\to\infty} \sup_{x\in B_N} \mathbf{P}_x \big( \capN{\mathcal{R}_{\infty} \cap B_N} \geq \tfrac{1}{\eta} \Theta_N \big) \leq 
C' \big( \eta \log \tfrac1\eta \big)^{d-2} \,.
\]
In dimension $d=3$, using that $\mathbf{P}(\mathscr{C}_W > \frac{1}{\eta\sqrt{A}}) \leq C A \eta^2$ since $\mathscr{C}_W$ has a finite second moment (see~\cite[Rem.~4.1]{Chang-cap-dim-3-4}), we get that
\[
\limsup_{N\to\infty} \sup_{x\in B_N} \mathbf{P}_x \big( \capN{\mathcal{R}_{\infty} \cap B_N} \geq \tfrac{1}{\eta} \Theta_N \big) \leq C  \inf_{A >0} \Big\{ A \eta^2 + A^{-1}\log A  \Big\} \leq C' \eta \log \tfrac{1}{\eta} \, ,
\]
the optimal being $A \approx \tfrac{1}{\eta} |\log \eta|^{1/2}$, thus concluding the proof of the first part of Proposition~\ref{lem:deviation-cap-infty}.

\subsubsection{Probability that \texorpdfstring{$\capN{\mathcal{R}_{\infty}\cap D_N}$}{} is small}

For $\eps>0$, let us introduce $D^{\eps} \defeq \{x \in D, d(x,D^c) \geq \eps\}$ and $D_N^{\eps} \defeq (N \cdot D^{\eps}) \cap \ZZ^d$.
Now, if we write $H_{N}^{\eps} \defeq H_{D_N^{\eps}} = \min\{i \, : \, S_i \in D_N^{\eps}\}$, we have that 
\[
\mathbf{P}_x \big(  \capN{\mathcal{R}_{\infty} \cap D_N} \leq \eta \Theta_N \big)
\leq \mathbf{P}_x\big( H_{N}^{\eps} = +\infty\big) + \mathbf{E}_x \Big[  \mathbf{P}_{S_{H_{N}^{\eps}}} \big(  \capN{\mathcal{R}_{\infty} \cap D_N} \leq \eta \Theta_N\big)  \indic{H_{N}^{\eps} <+\infty} \Big]\,.
\]
Now, since the ball centered at $S_{H_{N}^{\eps}}$ and of radius $\eps N$ is included in $D_N$, we get that 
\[
\begin{split}
	\mathbf{P}_{S_{H_{N}^{\eps}}} \big(  \capN{\mathcal{R}_{\infty} \cap D_N} \leq \eta \Theta_N\big) & \leq  \mathbf{P}_0 \big( \capN{\mathcal{R}_{\tau_{\eps N}}} \leq \eta \Theta_N\big) \\
	&\leq \mathbf{P}_0 \big( \capN{\mathcal{R}_{\eps^3 N^2}} \leq   \eta \Theta_N\big)  +  \mathbf{P}_0(\tau_{\eps N} \leq  \eps^3 N)\,,
\end{split}
\]
where $\tau_{r} \defeq \min\{i\geq 0, \|S_i\|\geq r\}$ is the exit time of the ball of radius $r$.

Now, by Donsker's invariance principle, we have that $\mathbf{P}_0(\tau_{r}\leq \eps r^2) \to \mathbf{P}_0( \sup_{s\in [0,1]} \|W_s\| \geq \sqrt{d/\eps})$ as $r\to\infty$, with $(W_s)_{s\geq 0}$ a standard $d$-dimensional Brownian motion.
By the reflection principle, we get as above that this probability is bounded by $2 e^{- \eps/2d}$.

Now, using Theorem~\ref{th:cv-cap-range} and recalling the definition~\eqref{eq:def-Theta} of $\Theta_N$, we also get that:
\begin{itemize}
	\item In dimension $d\geq 4$, if $\eps^3 = c_d \eta$ with $c_d > 1/\alpha_d$ if $d\geq 5$ and $c_d> 16/\pi^2$ if $d=4$, then 
	\[
	\lim_{N\to\infty} \mathbf{P}_0 \big( \capN{\mathcal{R}_{\eps^3 N^2}} \leq \eta \Theta_N\big) = 0 \, .
	\]
	
	\item In dimension $d=3$, taking $\eps=\eta^{1/3}$, we get that
	\[
	\lim_{N\to\infty} \mathbf{P}_0 \big( \capN{\mathcal{R}_{\eps^3 N^2}} \leq \eta \Theta_N\big) = \mathbf{P}\big(\mathscr{C}_W \leq \eta^{1/2} \big) \,.
	\]
\end{itemize}
All together, we obtain that for any $\eta>0$, letting $\eps \defeq c_d \eta^{1/3}$ for some appropriate constant, 
\[
\limsup_{N\to\infty} \mathbf{E}_x \Big[  \mathbf{P}_{S_{H_{N}^{\eps}}} \big(  \capN{\mathcal{R}_{\infty} \cap D_N} \leq \eta \Theta_N\big)   \indic{H_{N}^{\eps} <+\infty} \Big] \leq 2 e^{- c_d' \eta^{1/3}} + \mathbf{P}\big(\mathscr{C}_W \leq \eta^{1/2} \big) \indic{d=3} \,.
\]
Notice that $\mathbf{P}(\mathscr{C}_W  >0) =1$, so we get that $\mathbf{P}(\mathscr{C}_W \leq \eta^{1/2})$ goes to $0$ as $\eta$ goes to $0$. 

It therefore remains to bound
\[
\sum_{x\in D_N} \bar e_N(x) \mathbf{P}_x\big( H_N^{\eps} =+\infty \big)  = \frac{\capN{D_N} - \capN{D_N^{\eps}}}{\capN{D_N}}\,,
\]
where we have used the identity~\eqref{eq:diff-cap-inclusion}.
Now, we can use that the capacity of the blow-up $D_N = (N \cdot D) \cap \ZZ^d$ is related with the Newtonian capacity of $D$ via the following well-known result (see \cite[Lem. 2.2]{10.1214/aop/1176989003} or \cite[Prop. 2.4]{liLowerBoundDisconnection2014}).
\begin{proposition}\label{prop:capacite-blowup}
	For any compact regular set $D \subset \RR^d$, we have
	\[ \lim_{N \to +\infty} N^{2-d} \capN{D_N} =  \frac1d \capR{D} \, . \]
\end{proposition}

Therefore, we obtain that choosing $\eps = c_d \eta^{1/3}$ for some appropriate $c_d$,
\[
\limsup_{N\to\infty} \mathbf{P}_x \big(  \capN{\mathcal{R}_{\infty} \cap D_N} \leq \eta \Theta_N \big) \leq \frac{\capN{D} - \capN{D^{\eps}}}{\capN{D}} + 2 e^{- c_d' \eta^{1/3}} + \mathbf{P}\big(\mathscr{C}_W \leq \eta^{1/2} \big) \indic{d=3} \,.
\]
Since $\lim_{\eps\downarrow 0} \capN{D^{\eps}} = \capN{D}$ (see \cite[Prop 1.13]{port2012brownian}), we get that the right-hand side goes to $0$ as $\eta\downarrow 0$, which concludes the proof.
\qed

\section{Constrained random walk: Proof of Theorem \ref{th:ratio-cap-RW-boule}}\label{sec:capacite-CRW}

\subsection{Some preliminaries on confining probabilities}

Consider the transition kernel $P_A$ of the random walk killed on the boundary of a finite connected subset $A \subset \ZZ^d$. We introduce $(\lambda_A, \Phi_A)$ the first $\ell^1$-normalized eigencouple of $P_A$, that is
\begin{equation}
	P_A(x,y) = \frac{1}{2d} \indic{x \sim y; x,y \in A} \, , \quad P_A \Phi_A = \lambda_A \Phi_A \, , \quad \| \Phi_A \| = \sum_{z \in A} \Phi_A(z) = 1 \, .
\end{equation}
For $j \geq 2$ we also write $(\lambda_{A,j}, \Phi_{A,j})$ for the other ordered eigencouples, meaning that $(\lambda_{A,j})_{j \geq 2}$ is non-increasing and $1 > \lambda_A \geq \lambda_{A,2}$. Note that \cite[Lemma A.1]{dingDistributionRandomWalk2021a} proves the following bound on the maximal value of $\Phi_A$:
\begin{equation}\label{statement:borne-vect-propre}
	\text{There is a universal constant $\kappa > 0$ such that} \quad |\Phi_A|_\infty \leq \kappa (1 - \lambda_A)^{d/2}.
\end{equation}

When $A = B_N$ the discrete Euclidean ball of radius $N$, we will simply write $\lambda_N, \Phi_N$. Let us mention that it is known (see \textit{e.g.} \cite[(3.27),(6.11)]{weinberger1958lower}) that $\lambda_N$ satisfies
\begin{equation}\label{eq:encadrement-lambda}
	\lambda_N = 1 - \frac{\lambda}{2d} \frac{1}{N^2} \big( 1 + \grdO(\tfrac{1}{N}) \big) \, , 
\end{equation}
where $\lambda$ is the first eigenvalue of the Laplace-Beltrami operator on the unit ball of $\RR^d$.

In the proof, we will use the fact that the first eigencouple of the transition matrix restricted to $A$ is the main contributor to the probability for the random walk to stay inside $A$ for a long time. We will use the following result, which is can be proven similarly to Lemma 3.10 in \cite{dingDistributionRandomWalk2021a}.

\begin{lemma}\label{lem:eigentrucs}
	Let $D \subset \RR^d$ be a connected open set with smooth boundary, and write $D_N = (N \cdot D) \cap \ZZ^d$ for $N \geq 1$ large enough. Then, if $T \geq c N^2 \log N$ with $c$ large enough, for all $x \in D_N$ we have
	\begin{equation}
		\Pbf_x (\mathcal{R}_T \subseteq D_N) = (\lambda_{D_N})^{T} \frac{\Phi_{D_N}(x)}{\| \Phi_{D_N}^2 \|}(1 + \grdO(N^{-2})) \, .
	\end{equation}
	with $\grdO(N^{-2})$ uniform in $x \in D_N$ and $T \geq c N^2 \log N$.
\end{lemma}

We only give a sketch of the proof, as we will not use this exact result: we will restrict ourselves to the case where $D_N$ is the ball $B_N$.

\begin{proof}[Sketch of proof]
	The eigenvalues of the random walk absorbed on $\partial D_N$ satisfy an expansion similar to \eqref{eq:encadrement-lambda}. This implies a spectral gap of order $N^{-2}$. By writing the eigen-decomposition of the transition matrix, we use the spectral gap to show that the main contribution to the probability comes from $(\lambda_{D_N},\Phi_{D_N})$. Note that the assumption $t_N \gg N^2 \log N$ is used to absorb the secondary terms in the eigen-decomposition.
\end{proof}

In the case where the domain is the Euclidean ball, Lemma \ref{lem:eigentrucs} can be extended to all $T \geq N^2$ at the cost of a constant.

\begin{lemma}[{\cite[Corollary 6.9.6]{lawlerRandomWalkModern2010}}]\label{lem:rester-ds-boule}
	There exist a universal constant $c > 0$ such that for all $N$ large enough and all $T \geq N^2$,
	\begin{equation}
		\frac{1}{c} \lambda_N^{T} \leq \sup_{z \in B_N} \Pbf_{z}(\mathcal{R}_{T} \subseteq B_N) = \Pbf_{0}(\mathcal{R}_{T} \subseteq B_N) \leq c \lambda_N^{T} \, .
	\end{equation}
\end{lemma}

In the following, we will also study the case where the domain is an annulus, in which case using the local limit theorem yields a similar upper bound. More precisely, for any $K > 0$, there are positive constants $c_K>0$ and $C > 0$ such that, for all $T \geq 1$, for all sequences $a_N > b_N \to +\infty$ such that $\limsup\limits_{N \to +\infty} \tfrac{a_N}{b_N} \leq K < +\infty$, provided $N$ large enough we have
\begin{equation}\label{eq:proba-rester-ds-anneau}
	\sup_{z \in {B_{a_N} \setminus B_{b_N}}} \probaRW{z}{\mathcal{R}_{T} \subseteq B_{a_N} \setminus B_{b_N}} \leq C \exp \Big(- c_K \frac{T}{(a_N - b_N)^2} \Big) \, .
\end{equation}

To prove Theorem \ref{th:ratio-cap-RW-boule}, we will first prove a direct upper bound that goes to zero in the regime $t_N \Theta_N  N^{-d} \to 0$, meaning $t_N$ ``small''. Then, for the lower bound, we use the same strategy as in Section \ref{ssec:ratio-cap-key-argument} by considering the probability for the random walk conditioned to stay in $B_N$ of avoiding an infinite range. Fixing an $\eps > 0$ small enough, we get that with high probability, this infinite range leaves a trace $\mathcal{R}_\infty^\eps$ in $B_N$ of reasonnable size. In the regime $t_N \Theta_N  N^{-d} \to +\infty$, meaning $t_N$ is ``large'', we prove that avoiding $\mathcal{R}_\infty^\eps$ is too costly for the walk under $\Pbf( \cdot \, | \, \mathcal{R}_{t_N} \subseteq B_N)$.

\subsection{Upper bound on the ratio}

In this section, we prove that $\varsigma_{N,t_N}^{\mathrm{RW}} \to 0$ when $t_N N^{-d} \Theta_N \to 0$. The strategy is a direct upper bound on the capacity of the range $\mathcal{R}_{t_N}$ conditioned to stay in the ball. This upper bound relies on the fact that in this regime, time $t_N$ is sufficiently small, and thus the capacity of the walk does not have the time to get big enough.

\begin{proposition}\label{prop:tilt:borne-sup}
	Let $d \geq 3$, there are constants $c_1, c_2 > 0$ such that for all $N$ large enough we have
	\begin{equation}
		\Ebf [\cpc (\mathcal{R}_{t_N}) \, | \, \mathcal{R}_{t_N} \subseteq B_N ] \leq c_1 \frac{t_N}{N^2} \Theta_N + c_2 \Theta_{N \wedge \sqrt{t_N}} \, .
	\end{equation}
\end{proposition}

\noindent Using Proposition \ref{prop:tilt:borne-sup}, since $\capN{B_N} \geq c_d N^{d-2}$ we have
\begin{equation}
	\limsup_{N \to +\infty} \Ebf \left[ \varsigma_{N,t_N}^{\mathrm{RW}}  \, | \, \mathcal{R}_{t_N} \subseteq B_N \right] \leq \limsup_{N \to +\infty} \bigg( c_1 \frac{t_N}{N^d} \Theta_N + c_2' \frac{\Theta_{N \wedge \sqrt{t_N}}}{N^{d-2}} \bigg) \, .
\end{equation}
By the assumption $t_N \Theta_N/ N^d \to 0$, the first term goes to $0$. For the second term, in dimension $d \geq 4$ we use $\Theta_{N \wedge \sqrt{t_N}} \leq \Theta_N \leq N$ which implies $\Theta_{N \wedge \sqrt{t_N}} N^{2-d} \to 0$. For $d = 3$, the assumption $t_N \Theta_N N^{-d} \to 0$ can be rewritten as $t_N N^{-2} \to 0$. In particular we have $\Theta_{N \wedge \sqrt{t_N}} N^{2-d} = \Theta_{\sqrt{t_N}} N^{-1} = \sqrt{t_N} / N \to 0$.
Therefore, in both cases, we get the first part of Theorem \ref{th:ratio-cap-RW-boule} for $t_N \Theta_N N^{-d} \to 0$.

\begin{proof}[Proof of Proposition \ref{prop:tilt:borne-sup}]
	Provided $t_N \geq N^2$, we split the range into bits of length $N^2$ and for $i \in \mathset{1 , \dots , \lfloor t_N N^{-2} \rfloor}$, we write $\mathcal{R}^{(i)} \defeq \big\{ S_{(i-1)N^2 +1} , \dots , S_{iN^2} \big\}$. Writing $\bar{t}_N \defeq \lfloor t_N/N^2 \rfloor N^2$, by subbadditivity of the capacity,
	\begin{equation}\label{eq:decomp-bouts-cap-CRW}
		\Ebf \big[ \cpc (\mathcal{R}_{t_N}) \, | \, \mathcal{R}_{t_N} \subseteq B_N \big] \leq \sum_{i = 1}^{\lfloor t_N/N^2 \rfloor} \Ebf \big[ \cpc \mathcal{R}^{(i)} \, | \,  \mathcal{R}_{t_N} \subseteq B_N \big] + \Ebf \big[ \cpc \mathcal{R}([\bar{t}_N, t_N]) \, | \,  \mathcal{R}_{t_N} \subseteq B_N \big] \, .
	\end{equation}
	Note that if $t_N < N^2$, we can write \eqref{eq:decomp-bouts-cap-CRW} with an empty sum and we are left to bound $\Ebf \big[ \cpc \mathcal{R}([\bar{t}_N, t_N]) \, | \,  \mathcal{R}_{t_N} \subseteq B_N \big]$. Therefore, bounding \eqref{eq:decomp-bouts-cap-CRW} also covers the case where $t_N < N^2$.
	
	We first prove that $\Ebf [ \cpc \mathcal{R}^{(i)} \, | \,  \mathcal{R}_{t_N} \subseteq B_N ] \leq c \Theta_N$ for some $c > 0$ uniform in $i$. To do so, use the Markov property at times $(i-1) N^2$ and $i N^2$:
	\begin{equation}\label{eq:bouts-Markov}
	\begin{split}
		\Ebf \Big[ \cpc \big( \mathcal{R}^{(i)} \big) &\indic{\mathcal{R}_{t_N} \subseteq B_N} \Big] \\ &= \Ebf \Big[ \indic{\mathcal{R}_{iN^2} \subseteq D_N} \Ebf_{S_{iN^2}} \big[\cpc \big( \mathcal{R}_{N^2} \big) \indic{\mathcal{R}_{N^2} \subseteq B_N} \Pbf_{S_{(i+1)N^2}}(\mathcal{R}_{t_N - i N^2} \subseteq B_N) \big] \Big].
	\end{split}
	\end{equation}
	Using Lemma \ref{lem:rester-ds-boule}, we can bound the probability appearing in \eqref{eq:bouts-Markov}: there is a $c > 0$ such that
	\begin{equation}
		\Ebf \Big[ \cpc \big( \mathcal{R}^{(i)} \big) \indic{\mathcal{R}_{t_N} \subseteq B_N} \Big] \leq \Ebf \Big[ \indic{\mathcal{R}_{iN^2} \subseteq B_N} \Ebf_{S_{iN^2}} [\cpc \big( \mathcal{R}_{N^2} \big) ] \Big] c \lambda_N^{t_N - (i+1)N^2} \, ,
	\end{equation}
	where we also bounded the second indicator function by $1$.
	Since $\Pbf$ and $\cpc$ are invariant by translation, $\Ebf_{S_{iN^2}} [\cpc \mathcal{R}_{N^2}] = \Ebf [\cpc \mathcal{R}_{N^2}] \leq c \Theta_N$ by definition of $\Theta_N$ and Theorem \ref{th:cv-cap-range}. Again with Lemma \ref{lem:rester-ds-boule}, we finally get
	\begin{equation}
		\Ebf \Big[ \cpc \big( \mathcal{R}^{(i)} \big) \indic{\mathcal{R}_{t_N} \subseteq B_N} \Big] \leq c \Theta_N \lambda_N^{t_N - (i+1)N^2} \, ,
	\end{equation}
	Therefore, we deduce that
	\begin{equation}\label{eq:cap-bout-de-range}
		\Ebf [ \cpc \big( \mathcal{R}^{(i)} \big) \, | \,  \mathcal{R}_{t_N} \subseteq B_N ] \leq c \frac{\lambda_N^{t_N - N^2} \Theta_N}{\Pbf_0(\mathcal{R}_{t_N} \subseteq B_N)} \leq c' \Theta_N \ ,
	\end{equation}
	where we used that $\lambda_N^{- N^2}$ is bounded from above by a constant for $N$ large enough (recall \eqref{eq:encadrement-lambda}).
	
	Let us now turn to the last term in \eqref{eq:decomp-bouts-cap-CRW}, on which we use the Markov property in similar fashion as in \eqref{eq:bouts-Markov} to get
	\begin{equation}
		\begin{split}
			\Ebf \Big[ \cpc \big( \mathcal{R}([\bar{t}_N, t_N]) \big) \indic{\mathcal{R}_{t_N} \subseteq B_N} \Big] &\leq \Ebf \Big[ \indic{\mathcal{R}_{\bar{t}_N} \subseteq B_N} \Ebf_{S_{\bar{t}_N}} [\cpc \big( \mathcal{R}([\bar{t}_N, t_N]) \big) ] \Big] \\
			&\leq c \Theta_{(t_N - \bar{t}_N)^{1/2}} \Pbf(\mathcal{R}_{\bar{t}_N} \subseteq B_N) \, .
		\end{split}
	\end{equation}
	Lemma \ref{lem:rester-ds-boule} implies that $\Pbf(\mathcal{R}_{\bar{t}_N} \subseteq B_N) \leq c'' \Pbf(\mathcal{R}_{t_N} \subseteq B_N)$ for some $c'' > 0$.
	In particular, we get $\Ebf \Big[ \cpc \big( \mathcal{R}([\bar{t}_N, t_N]) \big) \indic{\mathcal{R}_{t_N} \subseteq B_N} \Big] \leq c'' \Theta_{(t_N - \bar{t}_N)^{1/2}}$.
	We have the immediate bound $\Theta_{(t_N - \bar{t}_N)^{1/2}} \leq \Theta_{N \wedge \sqrt{t_N}}$. Combining with \eqref{eq:cap-bout-de-range} and injecting in \eqref{eq:decomp-bouts-cap-CRW}, we get the announced result.
\end{proof}

\subsection{Lower bound for Theorem \ref{th:ratio-cap-RW-boule}}

In this section we assume $t_N$ to be \og{}large \fg{}, that is $t_N N^{-d} \Theta_N \to +\infty$, and prove that the constrained walk fills the ball $B_N$ in the sense of capacity. Similarly to the proof of Section \ref{ssec:ratio-cap-key-argument}, our starting point is identity \eqref{eq:diff-cap-inclusion} applied to $B = B_N$ and $A$ the range up to time $t_N$ of the confined random walk.

To discriminate between the confined random walk and the infinite range involved in \eqref{eq:diff-cap-inclusion}, we write $\Pbf^1_{x_1}$, $\Pbf^2_{x_2}$ for the laws of two independents random walks on $\ZZ^d$ starting at points $x_1,x_2 \in \ZZ^d$. The corresponding ranges $\mathcal{R}^1, \mathcal{R}^2$ respectively correspond to the (finite) range of the CRW, and the unconstrained (infinite) range.
Applying identity \eqref{eq:diff-cap-inclusion} with this notation yields
\begin{equation}\label{eq:ratio-CRW-decomp-range-infini}
	\Ebf^1_0 [1 - \varsigma_{t_N}^\mathrm{RW} \, | \, \mathcal{R}^1_{t_N} \subseteq B_N] = \sum_{z \in B_N} \bar{e}_{B_N}(z) \Ebf^1_0 [ \Pbf^2_z(\mathcal{R}^2_\infty \cap \mathcal{R}^1_{t_N} = \varnothing) \, | \, \mathcal{R}^1_{t_N} \subseteq B_N] \, .
\end{equation}

In order to obtain a lower bound on $\Ebf^1_0 \big[ \varsigma_{t_N}^\mathrm{RW} \, | \, \mathcal{R}^1_{t_N} \subseteq B_N \big]$, we therefore seek an upper bound on the expectation in the right-hand side of \eqref{eq:ratio-CRW-decomp-range-infini}, that is uniform in $z \in \partial B_N$.

We will use the same strategy as for the random interlacement by changing our point of view: note that
\begin{equation}
	\Ebf^1_0 [ \Pbf^2_z(\mathcal{R}^2_\infty \cap \mathcal{R}^1_{t_N} = \varnothing) \, | \, \mathcal{R}^1_{t_N} \subseteq B_N] = \Ebf^2_z[ \Pbf^1_0 (\mathcal{R}^2_\infty \cap \mathcal{R}^1_{t_N} = \varnothing \, | \, \mathcal{R}^1_{t_N} \subseteq B_N)] \, .
\end{equation}
Therefore, instead of evaluating the probability for the infinite range to avoid the constrained walk, we pick an infinite range and evaluate the probability for the contrained walk to avoid this range.

We start by proving that with high $\Pbf^2_z$-probability, $\mathcal{R}^2_\infty \cap B_N$ has a capacity that is at least of order $\Theta_N$. This will allow us to consider only the ranges $\mathcal{R}^2_\infty$ which are \og{}easy \fg{} to hit for the constrained walk. Afterwards, for such \og{}nice \fg{} configurations of $\mathcal{R}^2_\infty$, we prove it is too costly for the constrained walk to avoid $\mathcal{R}^2_\infty$ for such large time $t_N$.

Let us start by restricting ourselves to these \og{}nice \fg{} configurations of $\mathcal{R}^2_\infty$. For $r \in (0,1)$, we write $B_N^{r} \defeq B(0,r N) = \{x \in B_N, d(x,B_N^c) \geq (1-r) N \}$, and for $\eps > 0$ we define the stopping time $\tau_\eps \defeq H_{B_N^{1-3\eps}} \wedge H_{\partial B_N^{1-\eps}}$. Note that when starting from $z \in \partial B_N^{1-2\eps}$, the time $\tau_\eps$ is the exit time of the annulus $B_N^{1-\eps} \setminus B_N^{1-3\eps}$. Let $\eps, \eta > 0$, we then define the event
\begin{equation}
	\mathcal{A}_{N,2}^{\eps, \eta} \defeq \mathset{H^2_{B^{1-2\eps}_N} < +\infty, \capN{\mathcal{R}^2 \big( [H_{B_N^{1-2\eps}}, H_{B_N^{1-2\eps}} + \tau_\eps] \big)} \geq \eta \Theta_{\eps N}} \, .
\end{equation}
On the event $\mathcal{A}_{N,2}^{\eps, \eta}$, the infinite range $\mathcal{R}^2_\infty$ starting at $\partial B_N$ reaches $B_N^{1-2\eps}$, and its range before exiting the annulus $B_N^{1-\eps} \setminus B_N^{1-3\eps}$ has capacity at least $\eta \Theta_{\eps N}$. On the event $\big\{ H^2_{B^{1-2\eps}_N} < +\infty \big\}$, we write $\mathcal{R}^{2,\eps}_{\infty} \defeq \mathcal{R}^2 \big( [H^2_{B^{1-2\eps}_N}, H^2_{B^{1-2\eps}_N} + \tau_\eps] \big)$ this part of the range.

\begin{lemma}\label{lem:proba-mauvaise-config}
	Let $d \geq 3$ and fix $\eps \in (0,\tfrac14)$. Then, for any $\eta > 0$ sufficiently small, there is a $\delta(\eps, \eta)$ that goes to $0$ as $\eta$ and $\eps$ go to $0$ such that
	\begin{equation}
		\limsup_{N \to +\infty} \sum_{z \in B_N} \bar{e}_{B_N}(z) \Pbf^2_z \big( (\mathcal{A}_{N,2}^{\eps, \eta})^c \big) \leq \delta(\eps, \eta) \, .
	\end{equation}
\end{lemma}

\begin{proof}
	We first get a bound on the probability for $\mathcal{R}^2_\infty$ to avoid $B^{1-2\eps}_N$: using the identity \eqref{eq:diff-cap-inclusion},
	\begin{equation}\label{eq:Rinfty-penetre-Bn}
		\sum_{z \in B_N} \bar{e}_{B_N}(z) \Pbf^2_z (H^2_{B^{1-2\eps}_N} = +\infty) = 1 - \frac{\cpc \big(B^{1-2\eps}_N \big)}{\cpc(B_N)} \leq c_1 \eps \, ,
	\end{equation}
	for some $c_1 > 0$ that is uniform in $N$. Note that on the event $\mathset{\tau_\eps > \eps^3 N^2}$, we have $\cpc (\mathcal{R}^2_{\tau_\eps}) \geq \cpc (\mathcal{R}^2_{\eps^3 N^2})$. Therefore, for any $w \in \partial B_N^{1-2\eps}$, we have
	\begin{equation}
		\Pbf^2_w \Big( \cpc (\mathcal{R}^2_{\tau_\eps}) \leq \eta \Theta_{\eps N} \Big) \leq \Pbf^2_w \Big( \cpc (\mathcal{R}^2_{\eps^3 N^2}) \leq \eta \Theta_{\eps N} \Big) + \Pbf^2_w (\tau_\eps \leq \eps^3 N^2) \, .
	\end{equation}
	Define $\tilde{A}^\eps = \{ y \in \RR^d \, : \, 1-3\eps \leq |y| \leq 1- \eps \}$ and write $\tilde{\tau}_\eps$ for the exit time of $\tilde{A}^\eps$ for the Brownian motion. With the same arguments as above Proposition \ref{prop:capacite-blowup}, we have the following upper bound:
	\begin{equation}\label{eq:laisser-une-capacite-ok}
		\varlimsup_{N \to +\infty} \sup_{z \in B_N} \Pbf^2_z \Big( \frac{\capN{\mathcal{R}^{2,\eps}_{\infty}}}{\Theta_{\eps N}} \leq \eta, H_{B_N^{1-2\eps}} < +\infty \Big) \leq \kappa_d(\eps \eta) + \sup_{|x| = 1-2\eps} \Pbf_x (\tilde{\tau}_\eps \leq \eps^3) \, ,
	\end{equation}
	with $\kappa_d(\eta) \defeq \bP\big(\mathscr{C}_W \leq \eta^{1/2} \big)$ if $d=3$ (with $\mathscr{C}_W$ as defined in~\eqref{eq:cv-cap-range-3}), and $\kappa_d(\eta) = 0$ if $d \geq 4$. The right-hand side of \eqref{eq:laisser-une-capacite-ok} goes to $0$ as $\eta$ and $\eps$ go to $0$ for the same reasons as in the second part of Section~\ref{ssec:control-cap-range-infty-RI}. Since
	\[ \Pbf^2_z \big( (\mathcal{A}_{N,2}^{\eps, \eta})^c \big) \leq \Pbf^2_z (H^2_{B^{1-2\eps}_N} = +\infty) + \Pbf^2_z \Big(\frac{\capN{\mathcal{R}^2_{\tau_\eps}}}{\Theta_{\eps N}} \leq \eta, H_{B_N^{1-2\eps}} < +\infty \Big)  \, , \]
	assembling \eqref{eq:Rinfty-penetre-Bn} and \eqref{eq:laisser-une-capacite-ok} yields the lemma.
\end{proof}

To prove Theorem \ref{th:ratio-cap-RW-boule} in the case $t_N \Theta_N / N^d \to +\infty$, we will prove the following result.

\begin{proposition}\label{prop:proba-avoid-bon-event}
	For $d \geq 3$, and any $\eps, \eta > 0$, assuming $t_N \Theta_N / N^d \to +\infty$ we have
	\begin{equation}
		\limsup_{N \to +\infty} \sum_{z \in B_N} \bar{e}_{B_N}(z) \Ebf^2_z \Big[ \Pbf^1_0 (\mathcal{R}^2_\infty \cap \mathcal{R}^1_{t_N} = \varnothing \, | \, \mathcal{R}^1_{t_N} \subseteq B_N) \mathbbm{1}_{\mathcal{A}_{N,2}^{\eps, \eta}} \Big] = 0 \, .
	\end{equation}
\end{proposition}

Indeed, assuming Proposition \ref{prop:proba-avoid-bon-event}, we choose $\eps > 0$ and $\eta > 0$, and recall \eqref{eq:ratio-CRW-decomp-range-infini}. Then, combining Lemma \ref{lem:proba-mauvaise-config} and Proposition \ref{prop:proba-avoid-bon-event}, we have
\begin{equation}
	\liminf_{N \to +\infty} \Ebf^1_0 [\varsigma_{t_N}^\mathrm{RW} \, | \, \mathcal{R}^1_{t_N} \subseteq B_N] \geq 1 - \delta(\eps, \eta) \, .
\end{equation}
Taking $\eta$ and $\eps$ to zero then yields Theorem \ref{th:ratio-cap-RW-boule} in the case $t_N \Theta_N / N^d \to +\infty$.

\par A way to prove Proposition \ref{prop:proba-avoid-bon-event} is to relate the probability in the expectation to eigenvalues of the transition matrix of the random walk killed when hitting $\partial B_N \cup \mathcal{R}^2_\infty$. However, we are only able to prove this connection in the case where $t_N$ is at least $C N^2 \log N$ for a constant $C > 0$ large enough, which does not cover the full regime $t_N/ N^2 \to +\infty$ if $d = 3$.
Therefore, we will need to adopt another strategy for dimension three, which will fully use the fact that for $d \leq 3$, the random walk can easily hit another independent random walk (see \cite[Chapter 5]{lawlerIntersectionsRandomWalks1991} for example).

\subsubsection{Dimension four and higher}

On the event $\big\{ H^2_{B^{1-2\eps}_N} < +\infty \big\}$, recall the definition of $\mathcal{R}^{2,\eps}_{\infty}$ just above Lemma \ref{lem:proba-mauvaise-config}, and write $K = K_N^\eps$ for the connected component of $B_N \setminus \mathcal{R}^{2,\eps}_{\infty}$ that contains the origin. More simply put, $K_N^\eps$ contains all the points of $B_N \setminus \mathcal{R}^{2,\eps}_{\infty}$ that are accessible for the random walk started at $0$.
As explained previously, in dimension $d \geq 4$ we are interested in the case $t_N \gg N^2 \log N$. This means that the $\Pbf^1_0(\cdot \, | \, \mathcal{R}^1_{t_N} \subseteq B_N)$-probability for $\mathcal{R}^1_{t_N}$ to avoid $\mathcal{R}^{2,\eps}_{\infty}$ is given by the first eigenvalue of the domain $K_N^\eps$.

In the following, we will write $\lambda(\Lambda)$ to designate the first eigenvalue of the transition matrix $P_\Lambda$ of the simple random walk killed when exiting $\Lambda$.

\begin{lemma}\label{lem:rester-Kn}
	There is a constant $c > 0$ such that for $N$ large enough and $T \geq c' N^2 \log N$ with $c'$ large enough,
	\begin{equation}\label{eq:dim4+:proba-rester-K}
		\Pbf_0(\mathcal{R}_{T} \subseteq K_N^\eps) \leq c \lambda(K_N^\eps)^{T} \, .
	\end{equation}
	Moreover, there are positive constants $b, b_\eps, \mu$ such that for all $N$ large enough,
	\begin{equation}\label{eq:encadrement-lambdaK}
		1 - \frac{b_\eps}{N^2} \leq \lambda (B_0^{1-\eps}) \leq \lambda(K_N^\eps) \leq \lambda_N \leq 1 - \frac{b}{N^2} \quad , \qquad |\Phi_{K_N^\eps}|_\infty \leq \mu N^{-d} \, .
	\end{equation}
\end{lemma}

\begin{proof}
	We first prove \eqref{eq:encadrement-lambdaK}. The first part follows from $B_N^{1-3\eps} \subseteq K_N^\eps \subseteq B_N$ and \cite[Lemma A.2]{dingDistributionRandomWalk2021a}, which is an application of the minmax theorem for eigenvalues and asymptotic expansion for the eigenvalue of a discrete blowup. Combining this with \eqref{statement:borne-vect-propre} gives the second part.
	Recall that $\Phi_K$ stands for the first eigenvector of the transition matrix of simple random walk starting from $0$, killed on $\partial K$.
	The first part of the lemma is a consequence of the following claim: assuming $T \geq N^2 \log N$ and $N$ large enough, we have
	\begin{equation}\label{eq:dim4+:proba-rester-K-t_N}
		\Big| \Pbf_0(\mathcal{R}_{T} \subseteq K_N^\eps) - \lambda(K_N^\eps)^{T} \frac{\Phi_{K_N^\eps}(0)}{\| \Phi_{K_N^\eps} \|_2^2} \Big| \leq \lambda(K_N^\eps)^{T} |K_N^\eps| e^{-c \frac{T}{N^2}} \, .
	\end{equation}
	This claim follows from the same proof as in \cite[Lemma 3.10]{dingDistributionRandomWalk2021a} thanks to \eqref{eq:encadrement-lambdaK} that implies a spectral gap of order $N^{-2}$ for $K_N^\eps$ (see \cite[Lemma A.2]{dingDistributionRandomWalk2021a}).
	\par Notice that combining $|\Phi_{K_N^\eps}|_\infty \leq \mu N^{-d}$ with $\| \Phi_{K_N^\eps} \| = 1$ implies that there are at least $c_1 N^d$ points $z$ in $K_N^\eps$ such that $\Phi_{K_N^\eps}(z) \geq c_2 N^{-d}$, where $c_1, c_2$ are some small constants, independent from $N$ large enough. In particular, for any configuration of $K_N^\eps$,
	\[ \frac{\Phi_{K_N^\eps}(0)}{\| \Phi_{K_N^\eps} \|_2^2} \leq \mu N^{-d} \bigg[ \sum_{z \in K_N^\eps} \Phi_{K_N^\eps}^2(z) \mathbbm{1}_{\big\{\Phi_{K_N^\eps}(z) \geq c_2 N^{-d}\big\}} \bigg]^{-1} \leq \frac{\mu N^{-d}}{c_2^2 N^{-2d} c_1 N^d} = \frac{\mu}{c_1 c_2^2} \, . \]
	Assuming $T \geq c N^2 \log N$ with $c$ and $N$ large enough to have $|K_N^\eps| e^{-c \frac{T}{N^2}} \leq 1$, \eqref{eq:dim4+:proba-rester-K} implies
	\[
	\Pbf_0(\mathcal{R}_{T} \subseteq K_N^\eps) \leq 2 \lambda(K_N^\eps)^{T} \frac{\Phi_{K_N^\eps}(0)}{\| \Phi_{K_N^\eps}^2 \|} \leq \frac{2 \mu}{c_1 c_2^2} \lambda(K_N^\eps)^{T} \, ,
	\]
	thus proving \eqref{eq:dim4+:proba-rester-K} which concludes the proof of the lemma.
\end{proof}

\begin{lemma}\label{lem:eviter-via-vp}
	For any $\eps, \eta$ small enough, there exist positive constants $c, C$ such that for all $N$ large enough, all $z \in \partial B_N$,
	\begin{equation}
		\Ebf^2_z [ \Pbf^1_0(\mathcal{R}_{t_N} \subseteq K_N^\eps \, | \, \mathcal{R}_{t_N} \subseteq B_N) \mathbbm{1}_{\mathcal{A}_{N,2}^{\eps, \eta}}] \leq C \Ebf^2_z \left[ \exp \big( - c t_N (\lambda_N - \lambda(K_N^\eps)) \big) \mathbbm{1}_{\mathcal{A}_{N,2}^{\eps, \eta}} \right] \, .
	\end{equation}
\end{lemma}

\begin{proof}
	Note that using Lemmas \ref{lem:rester-ds-boule} \& \ref{lem:rester-Kn} (recall $t_N \gg N^2 \log N$), we get that for $n$ large enough,
	\begin{equation}\label{eq:lambda-vs-proba-Kn}
		\Pbf^1_0(\mathcal{R}_{t_N} \subseteq K_N^\eps \, | \, \mathcal{R}_{t_N} \subseteq B_N) = \frac{\Pbf^1_0(\mathcal{R}_{t_N} \subseteq K_N^\eps)}{\Pbf^1_0(\mathcal{R}_{t_N}\subseteq B_N)} \leq c_1 \bigg(\frac{\lambda(K_N^\eps)}{\lambda_N} \bigg)^{t_N} \, .
	\end{equation}
	Using \eqref{eq:encadrement-lambdaK}, we know that both $1 - \lambda_N$ and $1- \lambda(K_N^\eps)$ are of order $\asymp N^{-2}$. Therefore, a Taylor expansion of the logarithm implies that
	\[ \bigg(\frac{\lambda(K_N^\eps)}{\lambda_N} \bigg)^{t_N} = \exp \left( t_N [ \log \lambda_N - \log \lambda(K_N^\eps)] \right) \leq \exp \left( - c t_N (\lambda_N - \lambda(K_N^\eps)) \right) \, , \]
	for some constant $c > 0$ that is uniform in $N$ large enough. Since this holds true for any configuration of the range $\mathcal{R}^2_\infty$, we can take the expectation $\Ebf^2_z \big[ \, ( \, \cdot \, ) \, \mathbbm{1}_{\mathcal{A}_{N,2}^{\eps,\eta}} \big]$ on both sides.
\end{proof}

\begin{lemma}\label{lem:difference-vp-Dn-moins-range}
	There is a constant $c > 0$, independent from $N$, such that
	\begin{equation}
		\lambda_N - \lambda (K_N^\eps) \geq c \min \Big( N^{-d} \cpc(\mathcal{R}_\infty^{2,\eps}) \, ; \, N^{-2} \Big) \, .
	\end{equation}
\end{lemma}

\begin{proof}
	Using \eqref{eq:encadrement-lambdaK}, provided $\eps > 0$ small enough and $N \geq 1$ large enough, we have
	\begin{equation}
		\sum_{z \in K_N^\eps \setminus B_N^{1-3\eps}} \Phi_{K_N^\eps}(z) \leq \mu |K_N^\eps \setminus B_N^{1-3\eps}| N^{-d} \leq \mu |B_N \setminus B_N^{1-3\eps}| N^{-d} \leq \frac12 \, ,
	\end{equation}
	as well as $\sum_{z \in B_N \setminus B_N^{1-3 \eps}} \Phi_N(z) \leq \tfrac12$.
	From this, we deduce
	\begin{equation}\label{eq:sum-phiK-interieur}
		\sum_{z \in B_N^{1-3 \eps}} \Phi_{K_N^\eps}(z) \geq \frac12 \quad , \quad \sum_{z \in B_N^{1-3 \eps}} \Phi_N(z) \geq \frac12 \, .
	\end{equation}
	To end the proof, we will apply the proof of \cite[Lemma B.2]{dingDistributionRandomWalk2021a} which we quickly summarize. With the notation of \cite{dingDistributionRandomWalk2021a}, we take $\mathcal{B} = B_{R_1} = B_N$, $B_{R_2} = B_N^{1-\eps}$, $B_{R_3} = B_N^{1-3\eps}$ and $\mathcal{B}_o = K_N^\eps$. This ensures that $\mathcal{B} \setminus \mathcal{B}_o = \mathcal{R}_\infty^{2,\eps} \subseteq B_{R_2}$ and $B_{R_1} \subseteq \mathcal{B}$. Moreover, with \eqref{eq:sum-phiK-interieur} and $\lambda_{K_N^\eps} \geq \lambda_{B_{R_3}} \geq 1 - b N^{-2} = 1 - b' R_1^{-2}$, we see that all the assumptions of \cite[Lemma B.2]{dingDistributionRandomWalk2021a} hold in our case.
	
	For the sake of completeness, let us quickly summarize the main steps of the proof of \cite[Lemma B.2]{dingDistributionRandomWalk2021a}. First, \cite[(B.5)-(B.7)]{dingDistributionRandomWalk2021a} prove that there exists a constant $c_1 > 0$ such that
	\begin{equation}\label{eq:diff-vp-somme-bord}
		\lambda_{\mathcal{B}} - \lambda_{\mathcal{B}_o} \geq \frac{\min_{x \in B_{R_2}} \Phi_{\mathcal{B}}}{2d \| \Phi_{\mathcal{B}} \| \cdot \| \Phi_{\mathcal{B}_o} \|} \sum_{z \in \mathcal{B} \setminus \mathcal{B}_o} \Phi_{\mathcal{B}_o}(z) \geq c_1 \sum_{z \in \mathcal{B} \setminus \mathcal{B}_o} \Phi_{\mathcal{B}_o}(z) \, .
	\end{equation}
	Therefore, one only needs to get a lower bound on the sum in the right-hand side of \eqref{eq:diff-vp-somme-bord}. To do so, \cite[(B.10)]{dingDistributionRandomWalk2021a} combined with  gives
	\begin{equation}\label{eq:ding-somme-bord-cas-1}
		\sum_{z \in \mathcal{B} \setminus \mathcal{B}_o} \Phi_{\mathcal{B}_o}(z) \geq (1 - \lambda_\mathcal{B}) \sum_{x \in \mathcal{B}_o} \Phi_{\mathcal{B}_o}(x) \Pbf_x \big( S_{H_{\ZZ^d \setminus \mathcal{B}_o}} \in \mathcal{B} \setminus \mathcal{B}_o \big)
	\end{equation}
	Since $\lambda_\mathcal{B} \geq 1 - c R_1^{-2}$ (which stems from $\| \Phi_{\mathcal{B}_o} \|_\infty \geq c' R_1^{-d}$), we get that $\lambda_{\mathcal{B}} - \lambda_{\mathcal{B}_o} \geq c_2 R_1^{-2}$ provided that the sum on the right-hand side of \eqref{eq:ding-somme-bord-cas-1} is bounded from below by a constant.
	
	In the case where this sum is too small, it means in some way that the SRW can easily avoid touching $\mathcal{B} \setminus \mathcal{B}_o$ (note that this should be the case in dimension $d \geq 4$). What we get instead is \cite[(B.18)]{dingDistributionRandomWalk2021a} which reads
	\begin{equation}\label{eq:ding-somme-bord-cas-2}
		\sum_{z \in \mathcal{B} \setminus \mathcal{B}_o} \Phi_{\mathcal{B}_o}(z) \geq \frac{c_3}{R_1^d} \sum_{z \in \mathcal{B} \setminus \mathcal{B}_o} \Pbf_x \big( H_{\mathcal{B} \setminus \mathcal{B}_o} > H_{B_{R_1}} \big) \geq \frac{c_3}{R_1^d} \cpc \big( \mathcal{B} \setminus \mathcal{B}_o \big) \, .
	\end{equation}
	In conclusion, combining \eqref{eq:diff-vp-somme-bord} and \eqref{eq:ding-somme-bord-cas-1}-\eqref{eq:ding-somme-bord-cas-2} with our notations, the difference of eigenvalues $\lambda_N - \lambda (K_N^\eps)$ is bounded from below (up to a multiplicative constant) by either $R_1^{-2} = N^{-2}$ or $\cpc \big( \mathcal{B} \setminus \mathcal{B}_o \big) = \cpc \big( \mathcal{R}_\infty^{2,\eps} \big)$.
\end{proof}

\begin{proof}[Proof of Proposition \ref{prop:proba-avoid-bon-event} for $d \geq 4$]
	First using Lemma \ref{lem:eviter-via-vp}, we get the upper bound
	\begin{equation}
		\Ebf^2_z[ \Pbf^1_0 (\mathcal{R}^2_\infty \cap \mathcal{R}^1_{t_N} = \varnothing \, | \, \mathcal{R}^1_{t_N} \subseteq B_N) \mathbbm{1}_{\mathcal{A}_{N,2}^{\eps, \eta}}] \leq C \Ebf^2_z \big[ e^{- c t_N (\lambda_N - \lambda(K_N^\eps))} \mathbbm{1}_{\mathcal{A}_{N,2}^{\eps, \eta}} \big] \, .
	\end{equation}
	Then, with Lemma \ref{lem:difference-vp-Dn-moins-range},
	\begin{equation}
		\begin{split}
			\Ebf^2_z[ e^{- c t_N (\lambda_N - \lambda(K_N^\eps))} \mathbbm{1}_{\mathcal{A}_{N,2}^{\eps, \eta}}] &\leq \Ebf^2_z[ \exp \Big(- c' t_N  \min ( N^{-d} \cpc(\mathcal{R}_\infty^{2,\eps}) \, ; \, N^{-2}) \Big) \mathbbm{1}_{\mathcal{A}_{N,2}^{\eps, \eta}}]\\
			&\leq \exp \Big(-c \eta \frac{t_N}{N^d} \min ( \Theta_{\eps N} \, ; \, N^{d-2}) \Big) \, ,
		\end{split}
	\end{equation}
	where we used the definition of $\mathcal{A}_{N,2}^{\eps, \eta}$ to bound the capacity $\cpc(\mathcal{R}^{2,\eps}_{\infty})$ from below. Note that for any $\eps > 0$, provided $N$ large enough we always have $\Theta_{\eps N} \leq N^{d-2}$.
	Finally, since $\bar{e}_{B_N}$ is a probability measure on $\partial B_N$, we have
	\begin{equation}
		\sum_{z \in \partial B_N} \bar{e}_{B_N}(z) \Ebf^1_0 \big[ \Pbf^2_z(\mathcal{R}^2_\infty \cap \mathcal{R}^1_{t_N} = \varnothing, \mathcal{A}_{N,2}^{\eps, \eta}) \, | \, \mathcal{R}^1_{t_N} \subseteq B_N \big] \leq \exp \Big( -c \eta \frac{t_N}{N^d} \Theta_{\eps N} \Big) \, .
	\end{equation}
	By assumption, this upper bound goes to zero as $N \to +\infty$, thus proving the statement.
\end{proof}

\subsubsection{Dimension three}

As mentionned previously, in dimension three we want to take $t_N/N^2 \to +\infty$ thus we are not able to use Lemma \ref{lem:eviter-via-vp}. What we gain instead is that hitting a given set is easier than in higher dimensions. Fixing $\delta \in (0, \eps)$, the strategy of the proof consists in considering excursions of the conditioned walk from $B_N^{1-3\eps}$ to $B_N^{1-\delta}$. We prove in Lemma \ref{lem:toucher-range-excursion} that each of these excursions has a positive probability of hitting $\mathcal{R}^{2,\eps}_\infty$. By the Markov property, these excursions are independent and therefore the probability for $\mathcal{R}^1_{t_N}$ to avoid $\mathcal{R}^{2,\eps}_\infty$ is bounded by the product of probabilities. Lemma \ref{lem:transfo-laplace-nb-excursions} then proves that provided $t_N \gg N^2$, this product goes to $0$ as $N \to +\infty$, hence proving the theorem.

In this section, on the event $\mathcal{A}_{N,2}^{\eps,\eta}$, we fix a realization of the range $\mathcal{R}^{2,\eps}_\infty$, which is thus treated as any other subsets of $\ZZ^d$ with capacity at least $\Theta_{\eps N}$. In particular, hitting times refer to the other walk (with range $\mathcal{R}^1_{t_N}$). To lighten notation, we drop the \og{}$1$ \fg{} for the conditioned walk.

\begin{lemma}\label{lem:toucher-range-excursion}
	Let $d = 3$ and recall that $\mathcal{R}_\infty^{2,\eps} \subseteq B_N^{1-\eps} \setminus B_N^{1- 3\eps}$. Consider $\delta \in (0, \eps)$. There is a constant $c_0 = c_0(\eps,\eta, \delta) > 0$ such that uniformly in $N$ large enough, on the event $\mathcal{A}_{N,2}^{\eps,\eta}$ we have
	\begin{equation}
		\inf_{x \in B_N^{1-3\eps}} \Pbf_x \Big( H_{\mathcal{R}_\infty^{2,\eps}} < H_{\partial B_N^{1- \delta}} \Big) \geq c_0 \, .
	\end{equation}
\end{lemma}

\begin{proof}
	For $A \subseteq \ZZ^d$, write $G_A$ for the Green's function of the simple random walk absorbed on the interior boundary $\partial A$, that is $G_A(x,y) = \Ebf_x \big[\sum_{i = 0}^{H_{\partial A}} \indic{S_i = y} \big]$ for $x,y \in A$. We use a last exit decomposition for the random walk killed on $\partial B_N^{1-\eps}$:
	\begin{equation}
		\Pbf_x \Big( H_{\mathcal{R}_\infty^{2,\eps}} < H_{\partial B_N^{1-\delta}} \Big) = \sum_{y \in \mathcal{R}_\infty^{2,\eps}} G_{B_N^{1-\delta}}(x,y) \Pbf_y(H_{\mathcal{R}_\infty^{2,\eps}} > H_{\partial B_N^{1-\delta}}) \, .
	\end{equation}
	Now, we claim that there is a constant $c_\eps > 0$ such that for all $x \in B_N^{1-3\eps}$, all $y \in \mathcal{R}_\infty^{2,\eps}$, we have $G_{B_N^{1-\delta}}(x,y) \geq c_\delta/N$ for some $c_\delta > 0$ uniform in $N$ large enough. The proof can be found in the Appendix, see Lemma \ref{lem:LB-green-bulk-boule}.
	
	Since avoiding $\mathcal{R}_\infty^{2,\eps}$ before exiting $B_N^{1-\delta}$ can be done by avoiding it forever, we deduce that
	\begin{equation}
		\Pbf_x \Big( H_{\mathcal{R}_\infty^{2,\eps}} < H_{\partial B_N^{1-\delta}} \Big) \geq \frac{c_\delta}{N} \sum_{y \in \mathcal{R}_\infty^{2,\eps}} \Pbf_y(H_{\mathcal{R}_\infty^{2,\eps}} = +\infty) = \frac{c_\delta}{N} \cpc \big(\mathcal{R}_\infty^{2,\eps} \big) \, .
	\end{equation}
	On the event $\mathcal{A}_{N,2}^{\eps,\eta}$, we have $\cpc \big(\mathcal{R}_\infty^{2,\eps}\big) \geq \eta \Theta_{\eps N} = \eta \eps N$, and the lemma follows.
\end{proof}

Let $(\tau_i^{\mathrm{in}})_{i \geq 0}$ and $(\tau_i^{\mathrm{out}})_{i \geq 0}$ denote the successive returns to $B_N^{1-3\eps}$ after exiting $B_N^{1 - \delta}$, that is $\tau^{\mathrm{out}}_0 = 0$ and for $i \geq 0$,
\begin{equation}
	\tau^{\mathrm{in}}_{i + 1} = \inf \big\{ t > \tau^{\mathrm{out}}_i \, : \, X_t \in B_N^{1-3\eps} \big\} \quad , \quad \tau^{\mathrm{out}}_{i+1} = \inf \big\{ t > \tau^{\mathrm{in}}_{i+1} \, : \, X_t \not\in B_N^{1-2\delta} \big\} \, .
\end{equation}
Let us write $\mathcal{N}(t_N) \defeq \sup \mathset{i \, : \, \tau_i^{\mathrm{out}} \leq t_N}$, we then have
\begin{equation}
	\Pbf_x(\mathcal{R}_{t_N} \cap \mathcal{R}_\infty^{2,\eps} = \varnothing \, | \, \mathcal{R}_{t_N} \subseteq B_N) \leq  \Pbf_x(\forall i \in \mathset{1, \dots , \mathcal{N}(t_N)}, \mathcal{R}_{[\tau^{\mathrm{in}}_i, \tau^{\mathrm{out}}_i]} \cap \mathcal{R}_\infty^{2,\eps} = \varnothing \, | \, \mathcal{R}_{t_N} \subseteq B_N) \, .
\end{equation}
Note that by definition of the times $(\tau^{\mathrm{in}}_i)_{i \geq 1}$ and $(\tau^{\mathrm{out}}_i)_{i \geq 1}$, we have the equality between events:
\begin{equation}
	\mathset{\mathcal{R}_{t_N} \subseteq B_N} = \bigcap_{i = 1}^{\mathcal{N}(t_N)} \mathset{\mathcal{R}([\tau^{\mathrm{out}}_{i-1}, \tau^{\mathrm{in}}_{i}]) \subset B_N} \cap \mathset{\mathcal{R}([\tau^{\mathrm{out}}_{\mathcal{N}(t_N)}, t_N]) \subset B_N} \, .
\end{equation}
With Lemma \ref{lem:toucher-range-excursion}, we get that for all $i \geq 1$, $\Pbf_x(\mathcal{R}([\tau^{\mathrm{in}}_i, \tau^{\mathrm{out}}_i]) \cap \mathcal{R}_\infty^{2,\eps} = \varnothing) \leq 1-c_0$. Therefore, using the Markov property at times $(\tau^{\mathrm{in}}_i)_{i \geq 1}$ and $(\tau^{\mathrm{out}}_i)_{i \geq 1}$, we get the upper bound
\begin{equation}\label{eq:proba-esquive=transfo-laplace}
	\Pbf(\mathcal{R}_{t_N} \cap \mathcal{R}_\infty^{2,\eps} = \varnothing \, , \mathcal{R}_{t_N} \subseteq B_N) \leq \Ebf_x \left[ (1-c_0)^{\mathcal{N}(t_N)} \indic{\mathcal{R}_{t_N} \subseteq B_N} \right] \, .
\end{equation}

\begin{lemma}\label{lem:transfo-laplace-nb-excursions}
	Let $\eps > 0$ be small enough. For any sequence $(t_N)_{N \geq 1}$ which satisfies $t_N/N^2 \to +\infty$, we have the convergence
	\begin{equation}
		\lim_{N \to +\infty} \Ebf \left[ (1-c_0)^{\mathcal{N}(t_N)} \, | \, \mathcal{R}_{t_N} \subseteq B_N \right] = 0 \, .
	\end{equation}
\end{lemma}

\begin{proof}
	First write for any $K > 1$,
	\begin{equation}\label{eq:nb-moyen-excursion-UB}
		\lim_{N \to +\infty} \Ebf_x \left[ (1-c_0)^{\mathcal{N}(t_N)} \, \big| \, \mathcal{R}_{t_N} \subseteq B_N \right] \leq (1-c_0)^{K} + \lim_{N \to +\infty} \Pbf_x(\mathcal{N}(t_N) \leq K \, | \, \mathcal{R}_{t_N} \subseteq B_N ) \, .
	\end{equation}
	Hence, it suffices to prove that for any $K > 1$, the probability on the right-hand side of \eqref{eq:nb-moyen-excursion-UB} goes to $0$ as $N \to +\infty$. Recall that $\mathcal{N}(t_N)$ counts the number of excursions from $B_N^{1-3\eps}$ to $B_N^{1-\delta}$ achieved before time $t_N$. We can expect an excursion to take a time $\asymp N^2$ and thus have $\asymp t_N/N^2$ excursions. If $\mathcal{N}(t_N)$ is smaller then expected, this means that the walk spent too much time on one excursion.
	
	Note that $\mathcal{N}(t_N) \leq K$ implies that at least one of these excursions has length at least $t_N/K$, meaning that there exists a $i \in \llbracket 1, K \rrbracket$ such that $\tau^{\mathrm{in}}_i - \tau^{\mathrm{in}}_{i-1} \geq t_N/K$. For such $i$, this also implies that $\tau^{\mathrm{out}}_{i+1} - \tau^{\mathrm{in}}_i$ or $\tau^{\mathrm{in}}_{i+1} - \tau^{\mathrm{out}}_{i+1}$ is greater than $t_N/2K$. Finally, observe that writing $I_{k,K}^N \defeq \llbracket k t_N/4K, (k+1) t_N/4K \rrbracket$, all of the above imply that there is an interval of time $I_k^{K,(N)}$ such that $\mathcal{R}(I_k^{K,(N)}) \subseteq B_N^{1-\delta}$ or $\mathcal{R}(I_{k,K}^N) \subseteq B_N \setminus B_N^{1-\delta}$.
	
	If we summarize the previous paragraph, using two consecutive union bounds we get
	\begin{align}\label{eq:pas-assez-excursions-union-bound}
		\Pbf_x &(\mathcal{N}(t_N) \leq K \, | \, \mathcal{R}_{t_N} \subseteq B_N ) \notag \\
		&\leq \Pbf_x(\exists k \in \llbracket 0, 4K \rrbracket, \mathcal{R}(I_{k,K}^N) \subseteq B_N^{1- \delta} \text{ or } \mathcal{R}(I_{k,K}^N) \subseteq B_N \setminus B_N^{1-\delta} \, | \, \mathcal{R}_{t_N} \subseteq B_N) \\
		&\leq 2K \sup_{k \in \llbracket 0, 4K \rrbracket} \Big( \Pbf \big( \mathcal{R}(I_{k,K}^N) \subseteq B_N^{1- \delta} \, | \, \mathcal{R}_{t_N} \subseteq B_N \big) + \Pbf \big( \mathcal{R}(I_{k,K}^N) \subseteq B_N \setminus B_N^{1-\delta} \, | \, \mathcal{R}_{t_N} \subseteq B_N \big) \Big) \, . \notag
	\end{align}
	Using the same trick as in the proof of Proposition \ref{prop:tilt:borne-sup}, we use the Markov property at the start and the end of $I_{k,K}^N$ to control both of the conditional probabilities. For the first term, we get the upper bound
	\begin{equation}
		\Pbf(\mathcal{R}(I_{k,K}^N) \subseteq B_N^{1-2 \delta} \, | \, \mathcal{R}_{t_N} \subseteq B_N) \leq c_1 \lambda_N^{-t_N/4K} \sup_{z \in B_N^{1- \delta}} \Pbf_z(\mathcal{R}_{t_N/4K} \subseteq B_N^{1- \delta}) \, ,
	\end{equation}
	for which a use of Lemma \ref{lem:rester-ds-boule} yields
	\begin{equation}\label{eq:excursion-int-longue}
		\Pbf(\mathcal{R}(I_{k,K}^N) \subseteq B_N^{1- \delta} \, | \, \mathcal{R}_{t_N} \subseteq B_N) \leq c_1' \lambda_N^{(\frac{1}{(1-\delta)^2} - 1)t_N/4K} \leq c_1'' \exp \Big( -c_b 2\delta^2 \frac{t_N}{4KN^2} \Big) \, .
	\end{equation}
	On the other hand, with the same reasoning and a use of \eqref{eq:proba-rester-ds-anneau}, we get
	\begin{equation}\label{eq:excursion-ext-longue}
		\Pbf \Big(\mathcal{R}(I_{k,K}^N) \subseteq B_N \setminus B_N^{1-\delta} \, | \, \mathcal{R}_{t_N} \subseteq B_N) \leq c_2'' \exp \left( - (\tfrac{c_1}{\eps^2} - c_b) \frac{t_N}{4KN^2} \right) \, .
	\end{equation}
	Therefore, assuming that we took $\eps > 0$ small enough, injecting \eqref{eq:excursion-int-longue} and \eqref{eq:excursion-ext-longue} in \eqref{eq:pas-assez-excursions-union-bound} finally yields the bound
	\begin{equation}
		\Pbf_x(\mathcal{N}(t_N) \leq K \, | \, \mathcal{R}_{t_N} \subseteq B_N ) \leq c \exp \bigg( - \frac{c'}{4K} \frac{t_N}{N^2} \bigg)
	\end{equation}
	for some constants $c,c' > 0$ that do not depend on $N$ or $K$. In particular, we get the Lemma by using that $t_N N^{-d} \Theta_N \to +\infty$ means that in dimension three $t_N / N^2 \to +\infty$.
\end{proof}

\begin{remark}
	Note that this is the only proof that relies on the fact that we consider the random walk in a ball to get a lower bound in \eqref{eq:excursion-ext-longue} on the probability of staying inside $B_N$. If we consider a more general domain, we can apply Lemma \ref{lem:eigentrucs} to get the upper bound in \eqref{eq:excursion-ext-longue}, but applying the lemma works only for $t_N \gg N^2 \log N$. We believe however that Lemma \ref{lem:transfo-laplace-nb-excursions} should still hold for $t_N \gg N^2$, but its proof would require a more subtle approach.
\end{remark}

\begin{proof}[Proof of Proposition \ref{prop:proba-avoid-bon-event} for $d = 3$]
	Combining \eqref{eq:proba-esquive=transfo-laplace} with Lemma \ref{lem:transfo-laplace-nb-excursions} yields that on the event $\mathcal{A}^{\eps,\eta}_N$, for every configuration of $\mathcal{R}^{2,\eps}_\infty$,
	\[ \limsup_{N \to +\infty} \Pbf(\mathcal{R}_{t_N} \cap \mathcal{R}_\infty^{2,\eps} = \varnothing \, , \mathcal{R}_{t_N} \subseteq B_N) = 0 \, . \]
	Therefore, by noting again that
	\[ \Ebf^1_0 [ \Pbf^2_z(\mathcal{R}^2_\infty \cap \mathcal{R}^1_{t_N} = \varnothing, \mathcal{A}_{N,2}^{\eps, \eta}) \, | \, \mathcal{R}^1_{t_N} \subseteq B_N] = \Ebf^2_z[ \Pbf^1_0 (\mathcal{R}^2_\infty \cap \mathcal{R}^1_{t_N} = \varnothing \, | \, \mathcal{R}^1_{t_N} \subseteq B_N) \mathbbm{1}_{\mathcal{A}_{N,2}^{\eps, \eta}}] \, , \]
	we obtain that uniformly in $z \in \partial B_N$,
	\[ \limsup_{N \to +\infty} \Ebf^1_0 [ \Pbf^2_z(\mathcal{R}^2_\infty \cap \mathcal{R}^1_{t_N} = \varnothing, \mathcal{A}_{N,2}^{\eps, \eta}) \, | \, \mathcal{R}^1_{t_N} \subseteq B_N] = 0 \, . \]
	The uniformity in $z \in \partial B_N$ ensures that this holds if we integrate with respect to $\bar{e}_{B_N}$, hence proving the proposition.
\end{proof}

\section*{Appendix}

A classical result for the simple random walk states that the exit point of a ball is roughly uniform provided that the starting point is not too close to the boundary. Recall that $\tau_R$ is the exit time of the discrete Euclidean ball $B_R = B(0,R)$.

\begin{proposition}[{\cite[Lemma~6.3.7]{lawlerRandomWalkModern2010}}]\label{prop:SRW-sortie-unif}
	There is a constant $c_3 > 0$ such that, for any $R$ large enough,
	\begin{equation}
		\label{eq:SRW-sortie-boule}
		\frac{1}{c_3 R^{d-1}} \leq \inf_{u \in \bar B_{R/4}} \inf_{z \in \partial B_{R}} \Pbf_u\big( X_{\tau_R} = z \big) \leq \sup_{u \in \bar B_{R/4}} \sup_{z \in \partial B_{R}} \Pbf_u\big( X_{\tau_R} = z \big) \leq \frac{c_3}{R^{d-1}} \, .
	\end{equation} 
\end{proposition}

The proof relies on a use of a last exit decomposition on the discrete sphere $\partial B_{R/2}$ and time reversal of the last excursion.

An important consequence of Proposition \ref{prop:SRW-sortie-unif} is that given two points $x,y$ at typical distance $R$, the random walk has a positive probability of going from $x$ to a $\eta R$-neighborhood of $y$ while staying in a straight narrow corridor of width $\delta R$.

\begin{lemma}\label{lem:proba-approcher-dist-macro-point}
	Let $\eta, \delta > 0$, and for $\ell > 1$ write $\mathscr{R}_\ell(x,y) \defeq [x,y] + B(0,\ell)$ for the $\ell$ neighborhood of the segment $[x,y]$. Then, there is a constant $c_{\eta,\delta} > 0$ such that for any $R$ large enough,
	\begin{equation}
		\sup_{x,y ; |x-y| \leq R} \Pbf_x \Big( H_{B(y,\eta R)} < H_{\ZZ^d \setminus \mathscr{R}_{\delta R}(x,y)} \Big) \geq c_{\eta,\delta} \, .
	\end{equation}
\end{lemma}

\begin{proof}[Sketch of proof]
	Fix $x,y$. There is some $I = I(x,y) \in \NN$ such that there exists a chain $(z_i)_{i = 0}^{I}$ which satisfies $z_0 = x$, $z_{I} = y$, and for all $i \leq I-1$, we have $|z_{i+1}-z_i| = \delta R$ and $B(z_i,\delta R) \subset \mathscr{R}_{\delta R}(x,y)$. Then, there is a $c_\delta > 0$ such that for all $i$ we have $|\partial B(z_i, \delta R) \cap B(z_{i+1}, \tfrac14 \delta R)| \geq c_\delta R^{d-1}$. In particular, using Proposition \ref{prop:SRW-sortie-unif}, there is a $c_\delta' > 0$ such that for all $i$, \[ \Pbf_{z_i}(X_{H_{\partial B(z_i,\delta R)}} \in B(z_{i+1},\delta R/4)) \geq c_\delta' \, . \]
	Therefore, by the Markov property, $\Pbf_x(H_{B_R^{y,\delta}} < H_{\ZZ^d \setminus \mathscr{R}_{\delta R}(x,y)}) \geq (c_\delta')^{I}$. Note that we can bound from above the length $I(x,y)$ of the chain uniformly in $x$ and $y$, by some $I_0 = I_0(\delta)$. Therefore, we deduce that $\Pbf_x(H_{B(y,\delta R)} < H_{\ZZ^d \setminus \mathscr{R}_{\delta R}(x,y)}) \geq (c_\delta')^{I_0} = c(\delta) > 0$. Note that with such ``chain'' event, the SRW stays at distance at most $\delta R$ from the segment $[x,y]$, thus concluding the proof.
\end{proof}

With such estimate, we can get a lower bound on the Green function of the SRW killed on the boundary of a large ball $B_R$. The following lemma states that, restricting ourselves to the bulk of the ball, $G_{B_R}$ is comparable to $G$.

\begin{lemma}\label{lem:LB-green-bulk-boule}
	Fix $\eps > 0$. There is a constant $c_\eps > 0$ such that for all $R$ large enough, we have
	\begin{equation}
		\inf_{x,y \in B_{(1-\eps)R}} G_{B_R}(x,y) \geq c_\eps R^{2-d} \, .
	\end{equation}
	Note that we trivially have $G_{B_R}(x,y) \leq G(x,y) \leq C_d R^{2-d}$.
\end{lemma}

\begin{proof}
	Fix $\delta \in (0,\tfrac12 \eps)$ and consider the ball $B_{\delta R}^y$ centered at $y$ with radius $\delta R$. We have the lower bound
	\begin{equation}\label{eq:LB-green-boule}
		\begin{split}
			G_{B_R}(x,y) &\geq \Pbf_x(H_{B_{\delta R}^y} < H_{\partial B_R}) \inf_{z \in \partial B_{\delta R}^y} G_{B_R}(z,y) \\
			&\geq \Pbf_x \big( H_{B_{\delta R}^y} < H_{\partial B_R} \big) \inf_{z \in \partial B_{\delta R}} G_{B_{2\delta R}}(z,0) \, ,
		\end{split}
	\end{equation}
	where we have used the monoticity of $G_A$ with respect to the domain $A$ (since $B_{2 \delta R}^y \subseteq B_R$ by assumption on $\delta$) and the invariance of $\Pbf$ with respect to translations.
	Now, using \cite[Proposition 6.3.5]{lawlerRandomWalkModern2010}, we have for $R$ large enough:
	\[ G_{B_{\delta R}}(z,0) = a_d \big( |z|^{2-d} - (\delta R)^{2-d} \big) + \grdO(R^{1-d}) \geq \frac{a_d}{4 \delta^{d-2}} \frac{1}{R^{d-2}} \, . \]
	According to Lemma \ref{lem:proba-approcher-dist-macro-point} and since $\delta < \tfrac12 \eps$, the probability appearing in \eqref{eq:LB-green-boule} can be bounded from below by a constant $c = c(\delta)$ that is uniform in $x,y \in B_{(1-\eps)R}$, thus proving the lemma.
\end{proof}

	\bibliographystyle{agsm}
	\bibliography{capacite-RI-CRW-boule}
	
\end{document}